\documentclass[12pt]{article}
\usepackage{multicol}
\usepackage{amstext}
\usepackage[]{amsmath}
\usepackage{amsfonts}
\usepackage{amssymb}
\usepackage{amsthm}
\usepackage{newlfont}
\usepackage{graphicx}
\usepackage{sectsty}
\usepackage{lscape}
\usepackage[numbers,sort&compress]{natbib}
\usepackage[dvipsnames]{xcolor}
\newtheorem{thm}{Theorem}[section]
\newtheorem{ex}{Example}[section]
\newtheorem{lem}{Lemma}[section]

\theoremstyle{definition}
\newtheorem{defn}{Definition}[section]
\theoremstyle{remark}
\newtheorem{rem}{Remark}[section]
\numberwithin{equation}{section}
\begin{document}
\title{Further study on the conformable fractional Gauss hypergeometric function}
\date{}
\author{Mahmoud Abul-Ez$^{a}$\thanks{Corresponding author: mabulez56@hotmail.com},
Mohra Zayed$^{b}$\thanks{mzayed@kku.edu.sa } and Ali Youssef$^{c}$\thanks{alimohammedyouseuf@yahoo.com} \\ 
\footnotesize $^{a,c}$  Mathematics Department, Faculty of Science,
 Sohag University, Sohag 82524, Egypt.\\ 
\footnotesize $^{b}$ Mathematics Department, College of Science, 
 King Khalid University, Abha, Saudi Arabia.
}

\maketitle

\begin{abstract}
This paper presents a somewhat exhaustive study on the conformable fractional Gauss hypergeometric function (CFGHF). We start by solving the conformable fractional Gauss hypergeometric equation (CFGHE) about the fractional regular singular points $x=1$ and $x=\infty$. Next,  various generating functions of the CFGHF are established.  We also develop some differential forms for the CFGHF.
Subsequently, differential operators and the contiguous relations are reported.
Furthermore, we introduce the conformable fractional integral representation and the fractional Laplace transform of CFGHF.
As an application, and after making a suitable change of the independent variable,  we provide general solutions of some known conformable fractional differential equations, which could be written by means of the CFGHF.
\end{abstract}

 \vskip1mm 

\section{Introduction}
\noindent
We believe  that the old topics in classical analysis of special functions had been generalized to either the fractional calculus or the higher dimensional setting.
Fractional calculus has recently attracted considerable attention. It is defined as a generalization of differentiation and integration to an arbitrary order. It has become a fascinating branch of applied mathematics, which has recently stimulated mathematicians and physicists. Indeed, it represents a powerful tool to study a myriad of problems from different fields of science, such as statistical mechanics, control theory, signal and image processing, thermodynamics and quantum mechanics (see \cite{oldham1974fractional,samko1993fractional,miller1993introduction,podlubny1998fractional,hilfer2000applications,west2012physics,stankovic2006equation,tarasov2006fractional,magin2004fractional1,magin2004fractional2,gorenflo1997fractional,agrawal2004general,kilbas2006theory}).

Over the last four decades, several interesting and useful extensions of
many of the familiar special functions, such as the Gamma, Beta, and Gauss
hypergeometric functions were considered by various authors
of whom we may mention \cite{agarwal2014note,chaudhry1997extension,chaudhry2004extended,ozergin2011some,ozergin2011extension}. Functions of hypergeometric type constitute an important class of special functions.
The hypergeometric function $
_{2}F_{1}(\mu ,\nu ;c;x)$ plays a significant role in mathematical analysis and its
applications. This function allows one to solve many interesting
mathematical and physical problems, such as conformal mapping of triangular
domains bounded by line segments or circular arcs and various problems of
quantum mechanics. Most of the functions that occur in the analysis are classified as
special cases of the hypergeometric functions. Gauss first introduced and
studied hypergeometric series, paying particular attention to the cases when a
series converges to an elementary function which leads to study the hypergeometric series. Eventually, 
elementary functions and several other important functions in mathematics
can be expressed in terms of hypergeometric functions.
Hypergeometric functions can also be described as the solutions of special second-order linear
differential equations, which are  the hypergeometric differential equations. Riemann
was the first to exploit this idea and introduced a special symbol to
classify hypergeometric functions by singularities and exponents of
differential equations. The hypergeometric function is a solution of the
following Euler's hypergeometric differential equation 
\begin{equation}\label{eq1.1}
x\left( 1-x\right) \frac{d^{2}y}{dx^{2}}+\left[ c-\left( \mu +\nu +1\right) x %
\right] \frac{dy}{dx}-\mu \nu y=0,
\end{equation}
which has three regular singular points $0, 1,$ and $\infty$. The
generalization of this equation to three arbitrary regular singular points
is given by Riemann's differential equation. Any second order differential
equation with three regular singular points can be converted to the
hypergeometric differential equation by changing of variables.

Recently, as a conformable fractional derivative introduced of \cite{khalil2014new}, the authors in \cite{hammad2019fractional} used the new concept of fractional regular
singular points with the technique of fractional power series to solve the CFGHDE about $x=0$. They also  introduced
the form of the conformable fractional derivative and the integral
representation of the fractional Gaussian function. Besides, the solution of the fractional $
k-$hypergeometric differential equation was introduced in \cite{ali2020solution}.
As the Gauss hypergeometric differential equation appears in many
problems of physics, engineering, applied science, as well as finance and many
other important problems, it largely motivates us to conduct the present study.

In the present paper, we intend to continue the work
of Abu Hammad et al. \cite{hammad2019fractional} by  finding the solutions of  CFGHE about the fractional regular
singular points $x=1$ and $x=\infty .$ 
Afterward, we give a wide study on the CFGHF as follows. First, various generating functions of the CFGHF are established. Besides, some  differential forms are  developed for the CFGHF.
Then,  differential operators and  contiguous relations are derived. Furthermore, we introduce the conformable fractional integral
representation and the fractional Laplace transform of the CFGHF. As an application,
and after making a suitable change to the independent variable, we derive the 
general solutions of some conformable fractional differential equations (CFDE), which
could be written in terms of the  CFGHF. 
\section{Preliminaries and basic concepts}
Many definitions of fractional derivatives are obtained and compared.
The most popular definitions of fractional derivatives are due to Riemann-Liouvellie and Caputo.
It is pointed out that these definitions used an integral form
and lacked some basic properties, such as product rule, quotient rule, and
chain rule.

In 2014, Khalil et al. \cite{khalil2014new} introduced a surprising and satisfying definition of
the fractional derivative that is analog  to the classical derivative
definition called conformable fractional derivative (CFD). Their definition runs, as follows:
\begin{defn}
Let $f:\Omega \subseteq \left( 0,\infty \right) \rightarrow \mathbb{R} $ and 
$x\in \Omega .$ The conformable fractional derivative of order $\alpha \in
(0,1]$ for $f$ at $x$ is defined as 
\begin{equation*}
D^{\alpha }f\left( x\right) =\lim\limits_{h\rightarrow 0}\frac{f\left(
x+hx^{1-\alpha }\right) -f\left(x\right) }{h},
\end{equation*}
whenever the limit exists. The function $f$ is called $\alpha -$ conformable
fractional differentiable at $x.$ For $x=0,~D^{\alpha }f\left( 0\right)
=\lim\limits_{h\rightarrow 0^{+}}D^{\alpha }f\left( x\right) $ if such a limit
exists.
\end{defn}

This definition carries  very important and natural properties. Let $%
D^{\alpha }$ denote the conformable fractional derivative (CFD)  operator  of order $\alpha $. We recall from \cite{khalil2014new,abdeljawad2015conformable,hammad2016systems} some of its
general properties as follows.

Let $f$ and $g$ be $\alpha -$differentiable. Then, we have
\begin{itemize}
\item[(1)] Linearity: $D^{\alpha }\left( af+bg\right) \left( t\right)
=aD^{\alpha }f\left( t\right) +bD^{\alpha }g\left( t\right) ,$ For all$%
~a,b\in \mathbb{R}.$
\item[(2)] Product rule: $D^{\alpha }\left( fg\right) \left( t\right)
=f\left( t\right) D^{\alpha }g\left( t\right) +g\left( t\right) D^{\alpha
}f\left( t\right) .$
\item[(3)] Quotient rule: $D^{\alpha }\left( \frac{f}{g}\right) \left(
t\right) =\frac{g\left( t\right) D^{\alpha }f\left( t\right) -f\left(
t\right) D^{\alpha }g\left( t\right) }{g^{2}\left( t\right) },$ where $g\left( t\right) \neq
0.$
\item[(4)] Chain rule: $D^{\alpha }\left( f\circ g\right) \left( t\right)
=D^{\alpha }f\left( g\left( t\right) \right) .D^{\alpha }g\left( t\right)
.g\left( t\right) ^{\alpha -1}.$
\end{itemize}
Notice that for $\alpha =1$ in the $\alpha -$ conformable fractional
derivative, we get the corresponding classical limit definition of the
derivative. Also, a function could be $\alpha -$ conformable differentiable
at a point but not differentiable in the ordinary sense. For more details, we
refer to \cite{khalil2014new,abdeljawad2015conformable,el2020modification}.

Any linear homogenous differential equations of order two with
three regular singularities can be reduced to \eqref{eq1.1}. The hypergeometric function is
known as a solution to the hypergeometric equation \eqref{eq1.1}. One of the solutions of  the hypergeometric equation is given by the following
Gauss hypergeometric series in the form
\begin{equation}\label{eq2.1}
_{2}F_{1}\left( \mu ,\nu ;c;x\right) =\sum\limits_{n=0}^{\infty }\frac{\left(
\mu \right) _{n}\left( \nu \right) _{n}}{\left( c\right) _{n}}\frac{x^{n}}{n!}~\ \
\ \ \ \left( \left\vert x\right\vert <1\right)
\end{equation}
where $\left( b\right) _{n}$ stands for the usual Pochhammer symbol defined by 
\begin{equation*}
\left( b\right) _{n}=b\left( b+1\right) \left( b+2\right) ...\left( b+n-1\right) =\frac{\Gamma \left( b+n\right) }{\Gamma \left( b\right) 
},~\ n\in 
\mathbb{N}~\text{and }\left( b\right) _{0}=1.
\end{equation*}
Choosing the values of the parameters $\mu ,\nu $, and $c$ in an appropriately,
one can obtain many elementary and special functions as particular cases of
the Gauss hypergeometric series. For instance, the complete elliptic
integrals of the first and the second kinds, the Legendre associated
functions, ultra-spherical polynomials, and many others are special cases of
the function $_{2}F_{1}\left( \mu ,\nu ;c;x\right) .$
\begin{defn}\label{DCont} \cite{mubeen2014contiguous}
Two hypergeometric functions are said to be contiguous if their parameters $%
\mu ,\nu $, and $c$ differ by integers. The relations made by contiguous functions
are said to be contiguous function relations.
\end{defn}
\begin{defn}
The point $x=a$ is called an $\alpha -$regular singular point for the
equation
\begin{equation}\label{eq2.2}
D^{\alpha }D^{\alpha }y+P\left( x\right) D^{\alpha }y+Q\left( x\right) y=0,
\end{equation}
if $\lim\limits_{x\rightarrow a^{+}}(x^{\alpha }-a)P\left( x\right) $ and $%
\lim\limits_{x\rightarrow a^{+}}(x^{\alpha }-a)^{2}Q\left( x\right) $  exist.
\end{defn}
\begin{defn} 
\cite{abdeljawad2015conformable} A series $\sum\limits_{n=0}^{\infty }a_{n}x^{\alpha n}$  is called a
fractional Maclaurin power series.
\end{defn}
\begin{rem}
We will use $D^{n\alpha }$ to denote $\underset{n-\text{times}}{\underbrace{
D^{\alpha }D^{\alpha }...D^{\alpha }}}$. If $D^{n\alpha }f$ exists for all 
$n$ in some interval $[0,\lambda ]$ then one can write $f$ in the form of a
fractional power series
\end{rem}

\begin{defn}\label{DefFI} \cite{khalil2014new}
Suppose that $f:(0,\infty)\rightarrow \mathbb{R}$ is $\alpha$-differentiable, $\alpha \in (0,1]$, then  the  $\alpha$-fractional integral of $f$  is defined by
\begin{equation*}
I_{\alpha }^{a}f\left(t \right) = I_{1 }^{a}\left( t^{\alpha -1}f \right)=
  \int\limits_{a}^{t} \frac{f\left( x\right) }{x^{1-\alpha }}dx,\ t\geq 0.
\end{equation*}
\end{defn}

 For the infinite double series, we have the following useful Lemma (see \cite{rainville1969special}), which will be used in the sequel. 
 \begin{lem}\label{Lem1new}
 
 \begin{align}
 \sum_{n=0}^{\infty}\sum_{k=0}^{\infty}a_{k,n}&=\sum_{m=0}^{\infty}\sum_{j=0}^{m}a_{j,m-j}
 =\sum_{n=0}^{\infty}\sum_{k=0}^{n}a_{k,n-k}\\
 \sum_{n=0}^{\infty}\sum_{k=0}^{n}b_{k,n}&=\sum_{n=0}^{\infty}\sum_{k=0}^{\infty}b_{k,n+k}
 \end{align}
 \end{lem} 
 
\section{Solutions of the conformable fractional Gauss hypergeometric differential equation}
In our current study, we are interested to consider a generalization of
the differential equation \eqref{eq1.1} to fractional Gauss hypergeometric
differential equation, where the involving derivative is CFD. More precisely, we study the equation in the
form 
\begin{equation}\label{eq3.1}
x^{\alpha }\left( 1-x^{\alpha }\right) D^{\alpha }D^{\alpha }y+\alpha \left[
c-\left( \mu +\nu +1\right) x^{\alpha }\right] D^{\alpha }y-\alpha ^{2}\mu \nu y=0,
\end{equation}
where $\alpha \in (0,1]$ and $\mu ,\nu $ and $c$ are reals. \\
The new concept of fractional regular singular point together with the technique of fractional power series are used to solve the CFGHDE \eqref{eq3.1}. \\ 
Dividing \eqref{eq3.1} by $x^{\alpha }\left( 1-x^{\alpha }\right) $, we get
\begin{equation}\label{eq3.2}
D^{\alpha }D^{\alpha }y+\frac{\alpha \left\{ c-\left( \mu +\nu +1\right)
x^{\alpha }\right\} }{x^{\alpha }\left( 1-x^{\alpha }\right) }D^{\alpha }y-%
\frac{\alpha ^{2}\mu \nu }{x^{\alpha }\left( 1-x^{\alpha }\right) }y=0.
\end{equation}
Comparing \eqref{eq3.2} with \eqref{eq2.2}, we have
\begin{equation*}
P\left( x\right) =\frac{\alpha \left\{ c-\left( \mu +\nu +1\right)
x^{\alpha }\right\} }{x^{\alpha }\left( 1-x^{\alpha }\right) } \text{ and } Q\left(
x\right) =\frac{-\alpha ^{2}\mu \nu }{x^{\alpha }\left( 1-x^{\alpha }\right) 
}.
\end{equation*}
Clearly $x=0,~x=1$ and $x=\infty $ are $\alpha -$regular
singular points for \eqref{eq3.1}. \\
In 2020, the authors in \cite{hammad2019fractional} used the technique of fractional power series to
obtain the general solution of \eqref{eq3.1} about $x=0$ as
\begin{equation*}
y=A~_{2}F_{1}\left( \mu ,\nu ;c;x^{\alpha }\right) +B~x^{\alpha \left(
1-c\right) }~_{2}F_{1}\left( 1-c+\mu ,1-c+\nu ;2-c;x^{\alpha }\right),
\end{equation*}
where $A$ and $B$ are arbitrary constants and $_{2}F_{1}\left( \mu
,\nu ;c;x^{\alpha }\right) $ is CFGHF defined by
\begin{equation}\label{eq3.3}
_{2}F_{1}\left( \mu ,\nu ;c;x^{\alpha }\right) =\sum\limits_{n=0}^{\infty }%
\frac{\left( \mu \right) _{n}\left( \nu \right) _{n}}{\left( c\right) _{n}n!}%
x^{\alpha n};~~ \left\vert x^{\alpha }\right\vert <1.
\end{equation}
In fact, we will use a similar technique of \cite{hammad2019fractional}  to solve the equation \eqref{eq3.1} about the two $\alpha -$regular singular points $x=1$ and $x=\infty .$
\subsection{Solution of the CFGHE about $x=1$}
As $x=1$ is an $\alpha -$regular
singular point of \eqref{eq3.1}, therefore, the solution of \eqref{eq3.1} can be obtained in a series of powers of $\left( x^{\alpha }-1\right) $ as follows:

Taking $x^{\alpha }=1-t^{\alpha },$ this transfers the point $x=1$ to the point $t=0$
and therefore, we obtain the series solution of the following transformed
CFDE in terms of the series of powers of $ 
t^{\alpha }$:
\begin{equation}\label{eq3.5}
t^{\alpha }\left( 1-t^{\alpha }\right) D^{\alpha }D^{\alpha }y\left(
t\right) +\alpha \left\{ \left[ \mu +\nu +1-c\right] -\left( \mu +\nu
+1\right) t^{\alpha }\right\} D^{\alpha }y\left( t\right) -\alpha ^{2}\mu
\nu y\left( t\right) =0.
\end{equation}

Putting $c^{\prime }=\mu +\nu +1-c,$ in \eqref{eq3.5}, we get
\begin{equation}\label{eq3.6}
t^{\alpha }\left( 1-t^{\alpha }\right) D^{\alpha }D^{\alpha }y\left(
t\right) +\alpha \left\{ c^{\prime }-\left( \mu +\nu +1\right) t^{\alpha
}\right\} D^{\alpha }y\left( t\right) -\alpha ^{2}\mu \nu y\left( t\right) =0.
\end{equation}

This conformable fractional differential equation is similar to  CFGHE \eqref{eq3.1}.
So, the two linearly independent solutions of \eqref{eq3.6} can be stated in the form
\begin{equation}\label{eq3.7}
y_{1}=~_{2}F_{1}\left( \mu ,\nu ;c^{\prime };t^{\alpha }\right) \text{ and }
~y_{2}=t^{\alpha \left( 1-c\right) }~_{2}F_{1}\left( 1-c^{\prime }+\mu
,1-c^{\prime }+\nu ;2-c^{\prime };t^{\alpha }\right) 
\end{equation}

Now, replacing $c^{\prime }$ by $\left( \mu +\nu +1-c\right) $ and $
t^{\alpha }$ by $\left( 1-x^{\alpha }\right) $ in \eqref{eq3.7}, we get
\begin{equation*}
y_{1}=~_{2}F_{1}\left( \mu ,\nu ;\mu +\nu +1-c;t^{\alpha }\right)
\end{equation*}
and 
\begin{equation*}
y_{2}=\left( 1-x^{\alpha }\right) ^{\left( c-\mu -\nu \right)
}~_{2}F_{1}\left( c-\nu ,c-\mu ;c-\mu -\nu +1;1-x^{\alpha }\right) 
\end{equation*}
Thus the general solution of equation \eqref{eq3.1} about $x=1$ is given by
\begin{equation}\label{eq3.8}
\begin{split}
y=&A~_{2}F_{1}\left( \mu ,\nu ;\mu +\nu +1-c;t^{\alpha }\right) \\
&+B~\left(
1-x^{\alpha }\right) ^{\left( c-\mu -\nu \right) }~_{2}F_{1}\left( c-\nu
,c-\mu ;c-\mu -\nu +1;1-x^{\alpha }\right),
\end{split}
\end{equation}
where $A$ and $B$ are arbitrary constants.
\subsection{Solution of the CFGHE about $x=\infty $}
As $x=\infty $ is an $\alpha -$regular singular point of \eqref{eq3.1}, thus, the solution of \eqref{eq3.1} can be obtained in a series about $x=\infty $ by
putting  $x^{\alpha }=\frac{1}{\zeta ^{\alpha }}$ in \eqref{eq3.1}. Therefore, 
\begin{equation}\label{eq3.9}
D_{x}^{\alpha }~y=-\zeta
^{2\alpha }D_{\zeta }^{\alpha }~y\text{ \ and }D_{x}^{\alpha }D_{x}^{\alpha
}~y=\left[ 2\alpha \zeta ^{3\alpha }D_{\zeta }^{\alpha }~y+\zeta ^{4\alpha
}~D_{\zeta }^{\alpha }D_{\zeta }^{\alpha }~y\right]. 
\end{equation}

In view of \eqref{eq3.1}, we get 
\begin{equation}\label{eq3.10}
\zeta ^{2\alpha }\left( 1-\zeta ^{\alpha }\right) D_{\zeta }^{\alpha
}D_{\zeta }^{\alpha }~y+\alpha \left\{ 2\zeta ^{\alpha }\left( 1-\zeta
^{\alpha }\right) +c\zeta ^{2\alpha }-\zeta ^{\alpha }\left( \mu +\nu
+1\right) \right\} D_{\zeta }^{\alpha }~y\text{ }+\alpha ^{2}\mu \nu y=0
\end{equation}

Now, to find the solution, we proceed as follows.
Let $y=\sum\limits_{n=0}^{\infty }a_{n}\zeta ^{\alpha \left( s+n\right)
};~a_{0}\neq 0$ be the series solution of equation \eqref{eq3.10} about $\zeta =0.$ 
Then from the basic properties of the CFD we
get 
\begin{equation*}
D_{\zeta }^{\alpha }y=\sum\limits_{n=0}^{\infty }a_{n}~\alpha \left(
s+n\right) \zeta ^{\alpha \left( s+n-1\right) } \text{\ and }  
D_{\zeta }^{\alpha }D_{\zeta }^{\alpha }y=\sum\limits_{n=0}^{\infty
}a_{n}~\alpha ^{2}\left( s+n\right) \left( s+n-1\right) \zeta ^{\alpha
\left( s+n-2\right) }
\end{equation*}

Thus, owing to \eqref{eq3.10}, we have
\begin{equation*}
\begin{split}
\zeta ^{2\alpha }\left( 1-\zeta ^{\alpha }\right) &
\sum\limits_{n=0}^{\infty }a_{n}~\alpha ^{2}\left( s+n\right) \left(
s+n-1\right) \zeta ^{\alpha \left( s+n-2\right) }+\alpha \left\{ 2\zeta
^{\alpha }\left( 1-\zeta ^{\alpha }\right) +c\zeta ^{2\alpha }-\zeta
^{\alpha }\left( \mu +\nu +1\right) \right\}  \\
& \times \sum\limits_{n=0}^{\infty }a_{n}~\alpha \left( s+n\right) \zeta
^{\alpha \left( s+n-1\right) }+\alpha ^{2}\mu \nu \sum\limits_{n=0}^{\infty
}a_{n}\zeta ^{\alpha \left( s+n\right) }=0,
\end{split}
\end{equation*}
therefore,
\begin{equation*}
\begin{split}
&\sum\limits_{n=0}^{\infty }\alpha ^{2}a_{n}\left[ \left(
s+n\right) \left( s+n-1\right) +2~\left( s+n\right) -\left( \mu +\nu
+1\right) \left( s+n\right) +\mu \nu \right] \zeta ^{\alpha \left(
s+n\right) } \\
-& \sum\limits_{n=0}^{\infty }\alpha ^{2}a_{n}\left[ \left( s+n\right)
\left( s+n-1\right) +2~\left( s+n\right) -c~\left( s+n\right) \right] \zeta
^{\alpha \left( s+n+1\right) }=0.
\end{split}
\end{equation*}

Then we have
\begin{equation*}
\begin{split}
\alpha ^{2}& a_{0}\left[ s\left( s-1\right) +2~s-s\left( \mu
+\nu +1\right) +\mu \nu \right] \zeta ^{\alpha s} \\
+& \sum\limits_{n=1}^{\infty }\alpha ^{2}a_{n}\left[ \left( s+n\right)
\left( s+n-1\right) +2~\left( s+n\right) -\left( \mu +\nu +1\right) \left(
s+n\right) +\mu \nu \right] \zeta ^{\alpha \left( s+n\right) } \\
-& \sum\limits_{n=0}^{\infty }\alpha ^{2}a_{n}\left[ \left( s+n\right)
\left( s+n-1\right) +2~\left( s+n\right) -c~\left( s+n\right) \right] \zeta
^{\alpha \left( s+n+1\right) }=0.
\end{split}
\end{equation*}

A shift of index yields 
\begin{equation}\label{eq3.11}
\begin{split}
& \alpha ^{2}a_{0}\left[ s\left( s-1\right) +2~s-s\left( \mu +\nu +1\right)
+\mu \nu \right] \zeta ^{\alpha s} \\
& +\alpha ^{2}\sum\limits_{n=0}^{\infty }a_{n+1}\left[ \left( s+n+1\right)
\left( s+n\right) +2~\left( s+n+1\right) -\left( \mu +\nu +1\right) \left(
s+n+1\right) +\mu \nu \right]  \\
& -a_{n}\left[ \left( s+n\right) \left( s+n-1\right) +2~\left( s+n\right)
-c~\left( s+n\right) \right] \zeta ^{\alpha \left( s+n+1\right) }=0
\end{split}
\end{equation}

Equating the coefficients of $\zeta ^{\alpha s}$ to zero in \eqref{eq3.11}, we get the following indicial equation 
\begin{equation}\label{eq3.12}
s^{2}-s\left( \mu +\nu \right) +\mu \nu =0
\end{equation}

This equation \eqref{eq3.12} has two indicial roots $s=s_{1}=\mu $ and $
s=s_{2}=\nu $. \\
Again, equating to zero the coefficient of  $\zeta ^{\alpha (s+n+1)}$ in \eqref{eq3.11}, yields the recursion relation for $a_{n}$
\begin{equation}\label{eq3.13}
a_{n+1}=\frac{\left[ \left( s+n\right) \left( s+n-1\right) +2~\left(
s+n\right) -c~\left( s+n\right) \right] }{\left[ \left( s+n+1\right) \left(
s+n\right) +2~\left( s+n+1\right) -\left( \mu +\nu +1\right) \left(
s+n+1\right) +\mu \nu \right] }a_{n},
\end{equation}
or 
\begin{equation}\label{eq3.14}
a_{n+1}=\frac{\left( s+n\right) \left( s+n+1-c\right) }{\left( s+n+1\right) 
\left[ \left( s+n+1\right) -\mu -\nu \right] +\mu \nu }a_{n}
\end{equation}

To find the first solution of \eqref{eq3.10}, putting $s=\mu $ in \eqref{eq3.14}, we get 
\begin{equation*}
a_{n+1} =\frac{\left( \mu +n\right) \left( \mu +n+1-c\right) }{\left( \mu
+n+1\right) \left[ \left( n+1\right) -\nu \right] +\mu \nu }a_{n} 
=\frac{\left( \mu +n\right) \left( \mu +n+1-c\right) }{\left( n+1\right) 
\left[ \left( n+1\right) +\mu -\nu \right] }a_{n}.
\end{equation*}

Note that if $n=0,$ one can see
\begin{equation*}
a_{1}=\frac{\left( \mu \right) \left( \mu -c+1\right) }{\left[ \mu -\nu +1
\right] }a_{0}
\end{equation*}
and  for $n=1,$ we obtain
\begin{equation*}
a_{2} =\frac{\left( \mu +1\right) \left( \mu -c+2\right) }{\left[ \mu -\nu
+2\right] }a_{1} 
=\frac{\left( \mu \right) \left( \mu +1\right) \left( \mu -c+1\right)
\left( \mu -c+2\right) }{2\left[ \mu -\nu +1\right] \left[ \mu -\nu +2\right]
}a_{0}
\end{equation*}
Using the Pochhammer symbol we have 
\begin{equation*}
a_{2}=\frac{\left( \mu \right) _{2}\left( \mu -c+1\right) _{2}}{2!\left( \mu
-\nu +1\right) _{2}}a_{0}.
\end{equation*}
In general, we may write 
\begin{equation}\label{eq3.15}
a_{n}=\frac{\left( \mu \right) _{n}\left( \mu -c+1\right) _{n}}{n!\left( \mu
-\nu +1\right) _{n}}a_{0}.
\end{equation}
Letting $a_{0}=A,$ the first solution $y_{1}$ is given by
\begin{equation*}
\begin{split}
y_{1} =&A\sum\limits_{n=0}^{\infty }\frac{\left( \mu \right) _{n}\left( \mu
-c+1\right) _{n}}{\left( \mu -\nu +1\right) _{n}}\frac{^{\zeta ^{\alpha
\left( \mu +n\right) }}}{n!} 
=A\zeta ^{\alpha \mu }~_{2}F_{1}\left( \mu ,\mu -c+1;\mu -\nu +1;\zeta
^{\alpha }\right)  \\
=&Ax^{-\alpha \mu }~_{2}F_{1}\left( \mu ,\mu -c+1;\mu -\nu +1;\frac{1}{%
x^{\alpha }}\right) 
\end{split}
\end{equation*}

To find the second solution of \eqref{eq3.10}, putting $s=\nu $ in \eqref{eq3.14}, we have 
\begin{equation*}
a_{n+1} =\frac{\left( \nu +n\right) \left( \nu +n+1-c\right) }{\left( \nu
+n+1\right) \left[ \left( n+1\right) -\mu \right] +\mu \nu }a_{n} 
=\frac{\left( \nu +n\right) \left( \nu +n+1-c\right) }{\left( n+1\right) %
\left[ \left( n+1\right) +\nu -\mu \right] }a_{n},
\end{equation*}
from which we get
\begin{equation*}
a_{1}=\frac{\left( \nu \right) \left( \nu -c+1\right) }{\left[ \nu -\mu +1%
\right] }a_{0}.
\end{equation*}

Thus
\begin{equation*}
a_{2} =\frac{\left( \nu +1\right) \left( \nu -c+2\right) }{\left[ \nu -\mu
+2\right] }a_{1} 
=\frac{\left( \nu \right) \left( \nu +1\right) \left( \nu -c+1\right)
\left( \nu -c+2\right) }{2\left[ \nu -\mu +1\right] \left[ \nu -\mu +2\right]
}a_{0}.
\end{equation*}

Again by Pochhammer symbol yields
\begin{equation*}
a_{2}=\frac{\left( \nu \right) _{2}\left( \nu -c+1\right) _{2}}{2!\left( \nu
-\mu +1\right) _{2}}a_{0}
\end{equation*}
and in general
\begin{equation}\label{eq3.16}
a_{n}=\frac{\left( \nu \right) _{n}\left( \nu -c+1\right) _{n}}{n!\left( \nu
-\mu +1\right) _{n}}a_{0}
\end{equation}

Putting $a_{0}=B,$ the second solution $y_{2}$ is given by
\begin{equation*}
\begin{split}
y_{2} =&B\sum\limits_{n=0}^{\infty }\frac{\left( \nu \right) _{n}\left( \nu
-c+1\right) _{n}}{\left( \nu -\mu +1\right) _{n}}\frac{^{\zeta ^{\alpha
\left( \nu +n\right) }}}{n!}
=B\zeta ^{\alpha \nu }~_{2}F_{1}\left( \nu ,\nu -c+1;\nu -\mu +1;\zeta
^{\alpha }\right)  \\
=&Bx^{-\alpha \nu }~_{2}F_{1}\left( \nu ,\nu -c+1;\nu -\mu +1;\frac{1}{%
x^{\alpha }}\right) 
\end{split}
\end{equation*}

Therefore, the general solution of \eqref{eq3.1} about $x=\infty $ is 
\begin{equation*}
y=Ax^{-\alpha \mu }~_{2}F_{1}\left( \mu ,\mu -c+1;\mu -\nu +1;\frac{1}{%
x^{\alpha }}\right) +Bx^{-\alpha \nu }~_{2}F_{1}\left( \nu ,\nu -c+1;\nu
-\mu +1;\frac{1}{x^{\alpha }}\right) 
\end{equation*}
where $A$ and $B$ are arbitrary constants.
\begin{rem}
It is worthy to mention that the presented CFDE \eqref{eq3.10} is distinct to that one which was treated in \cite{hammad2019fractional,ali2020solution}. In fact \eqref{eq3.10} extended the Gauss hypergeometric differential equation given in \cite{rao2018generalized} to the conformable fractional context.
\end{rem}
\section{Generating functions}
Generating functions  are important way to transform formal power series into
functions and to analyze asymptotic properties of sequences. In what follows we characterize the CFGHF by means of various generating functions.
\begin{thm}\label{Gt1}
For $\alpha \in (0,1],$ the following generating function holds true
\begin{equation}\label{eq4.1}
\sum\limits_{m=0}^{\infty }\left( \mu \right) _{m}~~_{2}F_{1}\left( \mu
+m,\nu ;c;x^{\alpha }\right) .\frac{t^{\alpha m}}{m!}=\left( 1-t^{\alpha
}\right) ^{-\mu }~_{2}F_{1}\left( \mu ,\nu ;c;\frac{x^{\alpha }}{1-t^{\alpha
}}\right),
\end{equation}
where $ \left\vert x^{\alpha }\right\vert <1,$ and $\left\vert
t^{\alpha }\right\vert <1$.
\end{thm}

\begin{proof}
For convenience, let $\Im $ denote the left-hand side of \eqref{eq4.1}.  In view of \eqref{eq3.3}, we have
\begin{equation}\label{eq4.2}
\Im =\sum\limits_{m=0}^{\infty }\left( \mu \right) _{m}~~\left\{
\sum\limits_{n=0}^{\infty }\frac{\left( \mu +m\right) _{n}\left( \nu \right)
_{n}}{\left( c\right) _{n}}\frac{x^{\alpha n}}{n!}\right\} .\frac{t^{\alpha
m}}{m!}.
\end{equation}

Changing the order of summations in \eqref{eq4.2} and make use of identity $
\left( \mu \right) _{m}\left( \mu +m\right) _{n}=\left( \mu \right)
_{m+n}=\left( \mu \right) _{n}\left( \mu +n\right) _{m},$ yields 
\begin{equation*}
\Im =\sum\limits_{n=0}^{\infty }\frac{\left( \mu \right) _{n}\left( \nu
\right) _{n}}{\left( c\right) _{n}}\frac{x^{\alpha n}}{n!}%
.\sum\limits_{m=0}^{\infty }\frac{\left( \mu +n\right) _{m}}{m!}~t^{\alpha
m}~.
\end{equation*}

Using  the equality   $\sum\limits_{m=0}^{\infty }\frac{\left( \mu
+n\right) _{m}}{m!}~t^{\alpha m}=\left( 1-t^{\alpha }\right) ^{-\left( \mu
+n\right) };~~\left\vert t^{\alpha }\right\vert <1$ and the definition \eqref{eq3.3} immediately leads to the required result.
\end{proof}

\begin{thm}\label{Gt2}
For $\alpha \in (0,1],$ we have the following relation
\begin{equation}\label{eq4.3}
\sum\limits_{m=0}^{\infty }\left( \mu \right) _{m}~~_{2}F_{1}\left( -m,\nu
;c;x^{\alpha }\right) .\frac{t^{\alpha m}}{m!}=\left( 1-t^{\alpha }\right)
^{-\mu }~_{2}F_{1}\left( \mu ,\nu ;c;\frac{-x^{\alpha }t^{\alpha }}{%
1-t^{\alpha }}\right) 
\end{equation}
where $  \left\vert x^{\alpha }\right\vert <1,$ and
$\left\vert t^{\alpha }\right\vert <1$
\end{thm}

\begin{proof}
For short, set  $\Im $ to denote the left-hand side of \eqref{eq4.3}. Using \eqref{eq3.3}, one gets
\begin{equation}\label{eq4.4}
\Im =\sum\limits_{m=0}^{\infty }\left( \mu \right) _{m}~\left\{
\sum\limits_{n=0}^{\infty }\frac{\left( -m\right) _{n}\left( \nu \right) _{n}%
}{\left( c\right) _{n}}~\frac{x^{\alpha n}}{n!}\right\} .\frac{t^{\alpha m}}{%
m!}
\end{equation}

Since $\left( -m\right) _{n}=0$ if $n>m,$ then we may write
\begin{equation*}
\Im =\sum\limits_{m=0}^{\infty }\sum\limits_{n=0}^{m}\frac{\left( \mu
\right) _{m}~}{m!}\frac{\left( -m\right) _{n}\left( \nu \right) _{n}}{\left(
c\right) _{n}n!}x^{\alpha n}t^{\alpha m}
\end{equation*}

Using the congruence relation $\left( -m\right) _{n}=\frac{\left( -1\right)
^{n}m!}{\left( m-n\right) !},$ we get
\begin{equation}\label{eq4.5}
\Im =\sum\limits_{m=0}^{\infty }\sum\limits_{n=0}^{m}~\frac{\left( -1\right)
^{n}\left( \mu \right) _{m}\left( \nu \right) _{n}}{\left( c\right)
_{n}\left( m-n\right) !n!}x^{\alpha n}t^{\alpha m}
\end{equation}

Using lemma \ref{Lem1new}, equation \eqref{eq4.5} becomes
\begin{equation}\label{eq4.6}
\Im =\sum\limits_{m=0}^{\infty }\sum\limits_{n=0}^{\infty }~\frac{\left(
-1\right) ^{n}\left( \mu \right) _{n+m}\left( \nu \right) _{n}}{\left(
c\right) _{n}m!n!}x^{\alpha n}t^{\alpha \left( n+m\right) }
\end{equation}

Changing the order of summations in \eqref{eq4.6} and make use of identity $%
\left( \mu \right) _{n+m}=\left( \mu \right) _{n}\left( \mu +n\right) _{m},$
we get
\begin{equation*}
\Im =\sum\limits_{n=0}^{\infty }~\frac{\left( \mu \right) _{n}\left( \nu
\right) _{n}}{\left( c\right) _{n}n!}\left( -x^{\alpha }t^{\alpha }\right)
^{n}.\sum\limits_{m=0}^{\infty }~\frac{\left( \mu +n\right) _{m}}{m!}%
t^{\alpha m}
\end{equation*}

Using the binomial relation $\sum\limits_{m=0}^{\infty }~\frac{\left( \mu
+n\right) _{m}}{m!}t^{\alpha m}=\left( 1-t^{\alpha }\right) ^{-\left( \mu
+n\right) },~\ \left( \left\vert t^{\alpha }\right\vert <1\right) ,$ it follows that
\begin{equation*}
\begin{split}
\Im &=\sum\limits_{n=0}^{\infty }~\frac{\left( \mu \right) _{n}\left( \nu
\right) _{n}}{\left( c\right) _{n}n!}\left( -x^{\alpha }t^{\alpha }\right)
^{n}.\left( 1-t^{\alpha }\right) ^{-\left( \mu +n\right) }=\left(
1-t^{\alpha }\right) ^{-\mu }\sum\limits_{n=0}^{\infty }~\frac{\left( \mu
\right) _{n}\left( \nu \right) _{n}}{\left( c\right) _{n}n!}\left( \frac{%
-x^{\alpha }t^{\alpha }}{1-t^{\alpha }}\right) ^{n} \\
&=\left( 1-t^{\alpha }\right) ^{-\mu }~_{2}F_{1}\left( \mu ,\nu ;c;\frac{%
-x^{\alpha }t^{\alpha }}{1-t^{\alpha }}\right).
\end{split}
\end{equation*}

\end{proof}

\begin{thm}\label{Gt3}
For $\alpha \in (0,1],$ the following generating relation is valid
\begin{equation}\label{eq4.7}
\begin{split}
\sum\limits_{m=0}^{\infty }\frac{\left( \mu \right) _{m}\left( \nu \right)
_{m}}{\left( c\right) _{m}}~_{2}F_{1}\left( \mu +m,\nu +m;c+m;x^{\alpha
}\right) .\frac{t^{\alpha m}}{m!} =&~_{2}F_{1}\left( \mu ,\nu ;c;x^{\alpha
}+t^{\alpha }\right) ,
\end{split}
\end{equation}
where  $ \left\vert x^{\alpha }\right\vert <1,$ $ \left\vert t^{\alpha
}\right\vert <1,$ and $\left\vert x^{\alpha }+t^{\alpha }\right\vert <1$.
\end{thm}

\begin{proof}
Let $\Im $ denote the left-hand side of \eqref{eq4.7}, then using \eqref{eq3.3}, we obtain 
\begin{equation*}
\Im =\sum\limits_{m=0}^{\infty }\frac{\left( \mu \right) _{m}\left( \nu
\right) _{m}}{\left( c\right) _{m}}\left\{ \sum\limits_{n=0}^{\infty }\frac{%
\left( \mu +m\right) _{n}\left( \nu +m\right) _{n}}{\left( c+m\right) _{n}}~%
\frac{x^{\alpha n}}{n!}\right\} .\frac{t^{\alpha m}}{m!}
\end{equation*}

With the help of  $\left( \mu \right) _{n+m}=\left( \mu \right)
_{m}\left( \mu +m\right) _{n},$ we get
\begin{equation*}
\Im =\sum\limits_{m=0}^{\infty }\sum\limits_{n=0}^{\infty }\frac{\left( \mu
\right) _{m+n}\left( \nu \right) _{m+n}}{\left( c\right) _{m+n}n!m!}%
x^{\alpha n}t^{\alpha m}.
\end{equation*}

Using lemma \ref{Lem1new}, we have 
\begin{eqnarray*}
\Im  &=&\sum\limits_{m=0}^{\infty }\sum\limits_{n=0}^{m}\frac{\left( \mu
\right) _{m}\left( \nu \right) _{m}}{\left( c\right) _{m}n!\left( m-n\right)
!}x^{\alpha n}t^{\alpha \left( m-n\right) } \\
&=&\sum\limits_{m=0}^{\infty }\frac{\left( \mu \right) _{m}\left( \nu
\right) _{m}}{\left( c\right) _{m}m!}\sum\limits_{n=0}^{m}\frac{m!}{n!\left(
m-n\right) !}x^{\alpha n}t^{\alpha \left( m-n\right) }.
\end{eqnarray*}

The binomial theorem immediately gives
\begin{eqnarray*}
\Im  &=&\sum\limits_{m=0}^{\infty }\frac{\left( \mu \right) _{m}\left( \nu
\right) _{m}}{\left( c\right) _{m}m!}\left( x^{\alpha }+t^{\alpha }\right)
^{m} \\
&=&~_{2}F_{1}\left( \mu ,\nu ;c;x^{\alpha }+t^{\alpha }\right) 
\end{eqnarray*}
as required.
\end{proof}
\section{Transmutation formulas and differential forms}
\subsection{Transmutation formulas}
\begin{thm}\label{th4.1}
For $\left\vert x^{\alpha}\right\vert <1$ and $\left\vert \frac{x^{\alpha }}{%
1-x^{\alpha }}\right\vert <1$, the following identity is satisfied
\begin{equation}\label{eq5.1}
_{2}F_{1}\left( \mu ,\nu ;c;x^{\alpha }\right) =\left( 1-x^{\alpha }\right)
^{-\mu }~_{2}F_{1}\left( \mu ,c-\nu ;c;\frac{-x^{\alpha }}{1-x^{\alpha }}%
\right). 
\end{equation}
\end{thm}

\begin{proof}
Consider 
\begin{equation*}
\left( 1-x^{\alpha }\right) ^{-\mu }~_{2}F_{1}\left( \mu ,c-\nu ;c;\frac{
-x^{\alpha }}{1-x^{\alpha }}\right)  =\sum\limits_{k=0}^{\infty }\frac{\left( -1\right) ^{k}\left( \mu \right)
_{k}\left( c-\nu \right) _{k}}{\left( c\right) _{k}k!}x^{\alpha k}\left(
1-x^{\alpha }\right) ^{-\left( k+\mu \right) }.
\end{equation*}

In view of the expansion $\left( 1-x^{\alpha }\right) ^{-\mu }=\sum\limits_{n=0}^{\infty}\frac{\left( \mu\right)_{n}}{n!} x^{\alpha n}; \left\vert x^{\alpha }\right\vert <1, $  we may write
\begin{equation*}
\left( 1-x^{\alpha }\right) ^{-\mu }~_{2}F_{1}\left( \mu ,c-\nu ;c;\frac{%
-x^{\alpha }}{1-x^{\alpha }}\right) =\sum\limits_{k=0}^{\infty
}\sum\limits_{n=0}^{\infty }\frac{\left( -1\right) ^{k}\left( \mu \right)
_{k}\left( c-\nu \right) _{k}\left( \mu +k\right) _{n}}{\left( c\right)
_{k}k!n!}x^{\alpha \left( k+n\right) }.
\end{equation*}

Using the identity $\left( \mu +k\right) _{n}=\left( \mu \right) _{k+n},$ it is obvious
\begin{equation}\label{eq5.2}
\left( 1-x^{\alpha }\right) ^{-\mu }~_{2}F_{1}\left( \mu ,c-\nu ;c;\frac{%
-x^{\alpha }}{1-x^{\alpha }}\right) =\sum\limits_{n=0}^{\infty
}\sum\limits_{k=0}^{\infty }\frac{\left( -1\right) ^{k}\left( \mu \right)
_{k+n}\left( c-\nu \right) _{k}}{\left( c\right) _{k}k!n!}x^{\alpha \left(
k+n\right) }
\end{equation}

In virtue of lemma \ref{Lem1new} and the fact $\left( -n\right) _{k}=\frac{\left(
-1\right) ^{k}n!}{\left( n-k\right) !}$ one easily gets
\begin{equation}\label{eq5.3}
\left( 1-x^{\alpha }\right) ^{-\mu }~_{2}F_{1}\left( \mu ,c-\nu ;c;\frac{%
-x^{\alpha }}{1-x^{\alpha }}\right) =\sum\limits_{n=0}^{\infty
}\sum\limits_{k=0}^{n}\frac{\left( -n\right) _{k}\left( c-\nu \right) _{k}}{%
\left( c\right) _{k}k!}\frac{\left( \mu \right) _{n}x^{\alpha n}}{n!}.
\end{equation}

Since $\left( -n\right) _{k}=0$ if $k>n,$ then \eqref{eq5.3} becomes
\begin{equation}\label{eq5.4}
\left( 1-x^{\alpha }\right) ^{-\mu }~_{2}F_{1}\left( \mu ,c-\nu ;c;\frac{%
-x^{\alpha }}{1-x^{\alpha }}\right) =\sum\limits_{n=0}^{\infty
}\sum\limits_{k=0}^{\infty }\frac{\left( -n\right) _{k}\left( c-\nu \right)
_{k}}{\left( c\right) _{k}k!}\frac{\left( \mu \right) _{n}x^{\alpha n}}{n!}.
\end{equation}

Since the inner sum on the right of \eqref{eq5.4} is a terminating hypergeometric series, then
\begin{equation}\label{eq5.5}
\left( 1-x^{\alpha }\right) ^{-\mu }~_{2}F_{1}\left( \mu ,c-\nu ;c;\frac{%
-x^{\alpha }}{1-x^{\alpha }}\right) =\sum\limits_{n=0}^{\infty
}~_{2}F_{1}\left( -n,c-\nu ;c;1\right) \frac{\left( \mu \right)
_{n}x^{\alpha n}}{n!}.
\end{equation}

Due to $_{2}F_{1}\left( -n,c-\nu ;c;1\right) =\frac{\left( \nu
\right) _{n}}{\left( c\right) _{n}}$, the proof is therefore completed.
\end{proof}

\begin{thm}\label{th4.2}
For $\left\vert x^{\alpha}\right\vert <1$, the following identity is true
\begin{equation}\label{eq5.6}
_{2}F_{1}\left( \mu ,\nu ;c;x^{\alpha }\right) =\left( 1-x^{\alpha }\right)
^{c-\mu -\nu }~_{2}F_{1}\left( c-\mu ,c-\nu ;c;x^{\alpha }\right) 
\end{equation}
\end{thm}
\begin{proof}
By using assertion of theorem \ref{th4.1} and assuming that  $y^{\alpha }=\frac{-x^{\alpha }}{1-x^{\alpha }},$ it follows that
\begin{equation}\label{eq5.7}
_{2}F_{1}\left( \mu ,c-\nu ;c;y^{\alpha }\right) =\left( 1-y^{\alpha
}\right) ^{-\left( c-\nu \right) }~_{2}F_{1}\left( c-\mu ,c-\nu ;c;\frac{%
-y^{\alpha }}{1-y^{\alpha }}\right).
\end{equation}

From the assumption, we have $x^{\alpha }=\frac{%
-y^{\alpha }}{1-y^{\alpha }}$ which gives
\begin{equation}\label{eq5.8}
_{2}F_{1}\left( \mu ,c-\nu ;c;\frac{-x^{\alpha }}{1-x^{\alpha }}\right)
=\left( 1-x^{\alpha }\right) ^{\left( c-\nu \right) }~_{2}F_{1}\left( c-\mu
,c-\nu ;c;x^{\alpha }\right).
\end{equation}

A combinations of \eqref{eq5.8} and \eqref{eq5.1}, the result follows.
\end{proof}
\subsection{Some differential forms}
Now, according to the notation $D^{\alpha n}$, and due
to the fact $D^{\alpha }x^{p}=px^{p-\alpha }$ with $\alpha \in (0,1]$ and 
$\left\vert x^{\alpha }\right\vert <1,$ we state some interesting  conformable fractional
differential formulas for  $_{2}F_{1}\left( \mu ,\nu ;c;x^{\alpha }\right) $ as follows

\begin{align}
D^{\alpha }~_{2}F_{1}\left( \mu ,\nu ;c;x^{\alpha }\right) &=\frac{\alpha \mu
\nu }{c}~_{2}F_{1}\left( \mu +1,\nu +1;c+1;x^{\alpha }\right) \label{eq5.9}\\
D^{n\alpha }~_{2}F_{1}\left( \mu ,\nu ;c;x^{\alpha }\right)& =\frac{\alpha
^{n}\left( \mu \right) _{n}\left( \nu \right) _{n}}{\left( c\right) _{n}}%
~_{2}F_{1}\left( \mu +n,\nu +n;c+n;x^{\alpha }\right) \label{eq5.10}\\
D^{n\alpha }~\left\{ x^{\alpha \left( \mu +n-1\right) }~_{2}F_{1}\left( \mu
,\nu ;c;x^{\alpha }\right) \right\} &=\alpha ^{n}\left( \mu \right)
_{n}~x^{\alpha \left( \mu -1\right) }~_{2}F_{1}\left( \mu +n,\nu
;c;x^{\alpha }\right)\label{eq5.11}\\
D^{n\alpha }~\left\{ x^{\alpha \left( c-1\right) }~_{2}F_{1}\left( \mu ,\nu
;c;x^{\alpha }\right) \right\} &=\alpha ^{n}\left( c-n\right) _{n}~x^{\alpha
\left( c-n-1\right) }~_{2}F_{1}\left( \mu ,\nu ;c-n;x^{\alpha }\right) \label{eq5.12}
\end{align} 
\begin{equation}\label{eq5.13}
\begin{split}
D^{n\alpha }&~\left\{ x^{\alpha \left( c-\mu +n-1\right) }~\left( 1-x^{\alpha
}\right) ^{\mu +\nu -c}~_{2}F_{1}\left( \mu ,\nu ;c;x^{\alpha }\right)
\right\} \\
&=\alpha ^{n}\left( c-\mu \right) _{n}~x^{\alpha \left( c-\mu
-1\right) }\left( 1-x^{\alpha }\right) ^{\mu +\nu -c-n}~_{2}F_{1}\left( \mu -n,\nu ;c;x^{\alpha }\right) 
\end{split}
\end{equation}
\begin{equation}\label{eq5.14}
\begin{split}
D^{n\alpha }&~\left\{ ~\left( 1-x^{\alpha }\right) ^{\mu +\nu
-c}~_{2}F_{1}\left( \mu ,\nu ;c;x^{\alpha }\right) \right\} \\
&=\frac{\alpha
^{n}\left( c-\mu \right) _{n}\left( c-\gamma \right) _{n}}{\left( c\right)
_{n}}~\left( 1-x^{\alpha }\right) ^{\mu +\nu -c-n}~_{2}F_{1}\left( \mu ,\nu ;c+n;x^{\alpha }\right) 
\end{split}
\end{equation}
\begin{equation}\label{eq5.15}
\begin{split}
D^{n\alpha }&~\left\{ x^{\alpha \left( c-1\right) }~\left( 1-x^{\alpha
}\right) ^{\mu +\nu -c}~_{2}F_{1}\left( \mu ,\nu ;c;x^{\alpha }\right)
\right\} \\
&=\alpha ^{n}\left( c-n\right) _{n}~x^{\alpha \left( c-n-1\right)
}\left( 1-x^{\alpha }\right) ^{\mu +\nu -c-n}~_{2}F_{1}\left( \mu -n,\nu -n;c-n;x^{\alpha }\right) 
\end{split}
\end{equation}
\begin{equation}\label{eq5.16}
\begin{split}
D^{n\alpha }&~\left\{ x^{\alpha \left( n+c-1\right) }~\left( 1-x^{\alpha
}\right) ^{n+\mu +\nu -c}~_{2}F_{1}\left( \mu +n,\nu +n;c+n;x^{\alpha
}\right) \right\} \\
&=\alpha ^{n}\left( c\right) _{n}~x^{\alpha \left(
c-1\right) }\left( 1-x^{\alpha }\right) ^{\mu +\nu -c}~_{2}F_{1}\left( \mu ,\nu ;c;x^{\alpha }\right) 
\end{split}
\end{equation}

Such formulas can be proved using the series expansions of $%
_{2}F_{1}\left( \mu ,\nu ;c;x^{\alpha }\right) $ as given in \eqref{eq3.3}. However, we are going to prove the validity of \eqref{eq5.11} and \eqref{eq5.14}, while the other formulas can be proved similarly.
First, note that 
\begin{equation*}
D^{n\alpha }~\left\{ x^{\alpha \left( \mu +n-1\right) }~_{2}F_{1}\left( \mu
,\nu ;c;x^{\alpha }\right) \right\} =\sum\limits_{k=0}^{\infty }\frac{\left(
\mu \right) _{k}\left( \nu \right) _{k}}{\left( c\right) _{k}k!}D^{n\alpha
}\left\{ x^{\alpha \left( \mu +n+k-1\right) }\right\} 
\end{equation*}

The action of the conformable derivative gives
\begin{equation*}
\begin{split}
D^{n\alpha }~\left\{ x^{\alpha \left( \mu +n-1\right) }~_{2}F_{1}\left( \mu
,\nu ;c;x^{\alpha }\right) \right\}  &=\sum\limits_{k=0}^{\infty }\frac{
\left( \mu \right) _{k}\left( \nu \right) _{k}}{\left( c\right) _{k}k!}
\alpha ^{n}\frac{\Gamma \left( \mu +n+k\right) }{\Gamma \left( \mu +k\right) 
}x^{\alpha \left( \mu +k-1\right) } \\
&=\alpha ^{n}\sum\limits_{k=0}^{\infty }\frac{\left( \mu \right)
_{n+k}\left( \nu \right) _{k}}{\left( c\right) _{k}k!}x^{\alpha \left( \mu
+k-1\right) }
\end{split}
\end{equation*}

Knowing that $\left( \mu \right) _{n+k}=\left( \mu \right) _{n}\left( \mu
+n\right) _{k},$ it can be seen 
\begin{equation*}
\begin{split}
D^{n\alpha }~\left\{ x^{\alpha \left( \mu +n-1\right) }~_{2}F_{1}\left( \mu
,\nu ;c;x^{\alpha }\right) \right\}  &=\alpha ^{n}\left( \mu \right)
_{n}x^{\alpha \left( \mu -1\right) }\sum\limits_{k=0}^{\infty }\frac{\left(
\mu +n\right) _{k}\left( \nu \right) _{k}}{\left( c\right) _{k}k!}x^{\alpha
k} \\
&=\alpha ^{n}\left( \mu \right) _{n}~x^{\alpha \left( \mu -1\right)
}~_{2}F_{1}\left( \mu +n,\nu ;c;x^{\alpha }\right) 
\end{split}
\end{equation*}
as required. 

In view of \eqref{eq5.6}, we have 
\begin{equation}\label{eq5.17}
D^{n\alpha }\left\{ \left( 1-x^{\alpha }\right) ^{\mu +\nu
-c}~_{2}F_{1}\left( \mu ,\nu ;c;x^{\alpha }\right) \right\}=D^{n\alpha }~_{2}F_{1}\left( c-\mu ,c-\nu ;c;x^{\alpha }\right) 
\end{equation}

Using \eqref{eq5.10}, we get
\begin{equation}\label{eq5.18}
D^{n\alpha }\left\{ \left( 1-x^{\alpha }\right) ^{\mu +\nu
-c}~_{2}F_{1}\left( \mu ,\nu ;c;x^{\alpha }\right) \right\} =\frac{\alpha
^{n}\left( c-\mu \right) _{n}\left( c-\nu \right) _{n}}{\left( c\right) _{n}}
~_{2}F_{1}\left( c-\mu +n,c-\nu +n;c+n;x^{\alpha }\right) 
\end{equation}

Return to \eqref{eq5.6}, we obtain 
\begin{equation}\label{eq5.19}
\begin{split}
D^{n\alpha }\left\{ ~\left( 1-x^{\alpha }\right) ^{\mu +\nu
-c}~_{2}F_{1}\left( \mu ,\nu ;c;x^{\alpha }\right) \right\}=&\frac{\alpha
^{n}\left( c-\mu \right) _{n}\left( c-\gamma \right) _{n}}{\left( c\right)
_{n}}~\left( 1-x^{\alpha }\right) ^{\mu +\nu -c-n} \\
&\times ~_{2}F_{1}\left( \mu ,\nu ;c+n;x^{\alpha }\right) 
\end{split}
\end{equation}
\begin{rem}
In case of $\mu =-n$ in \eqref{eq5.16}, we obtain 
\begin{equation}\label{eq5.20}
D^{n\alpha }\left\{ x^{\alpha \left( n+c-1\right) }~\left( 1-x^{\alpha
}\right) ^{\nu -c}~\right\} =\alpha ^{n}\left( c\right) _{n}~x^{\alpha
\left( c-1\right) }\left( 1-x^{\alpha }\right) ^{\nu -c-n}~_{2}F_{1}\left(
-n,\nu ;c;x^{\alpha }\right) 
\end{equation}
\end{rem}
\section{Differential operator and contiguous relations}
Following \cite{rainville1969special}, define the conformable fractional operator $\theta ^{\alpha }$ in the form
\begin{equation}\label{eq6.1}
\theta ^{\alpha }=\frac{1}{
\alpha }x^{\alpha }D^{\alpha }.
\end{equation}

This operator has the particularly pleasant
property that $\theta ^{\alpha }x^{n\alpha }=nx^{n\alpha },$ which makes it
handy to be used on power series. 
In this section, relying on definition \ref{DCont}, we establish several results concerning contiguous relations for the CFGHF. To achieve that, we have to prove the following lemma.
 
\begin{lem}\label{Lem6.1}
Let $\alpha \in (0,1],$ then the CFGHF $_{2}F_{1}\left( \mu ,\nu ;c;x^{\alpha }\right) $ satisfies the
following 
\begin{align}
\left( \theta ^{\alpha }+\mu \right) _{2}F_{1}\left( \mu ,\nu ;c;x^{\alpha
}\right) &=\mu _{2}F_{1}\left( \mu +1,\nu ;c;x^{\alpha }\right)\label{eq6.2}\\
\left( \theta ^{\alpha }+\nu \right) _{2}F_{1}\left( \mu ,\nu ;c;x^{\alpha
}\right) &=\nu _{2}F_{1}\left( \mu ,\nu +1;c;x^{\alpha }\right)\label{eq6.3}\\
\left( \theta ^{\alpha }+c-1\right) _{2}F_{1}\left( \mu ,\nu ;c;x^{\alpha
}\right) &=\left( c-1\right) _{2}F_{1}\left( \mu ,\nu ;c-1;x^{\alpha }\right)\label{eq6.4}
\end{align}
\end{lem}
\begin{proof}
Using \eqref{eq3.3} and \eqref{eq6.1} it follows that
\begin{equation*}
\begin{split}
\left( \theta ^{\alpha }+\mu \right) _{2}F_{1}\left( \mu ,\nu ;c;x^{\alpha
}\right) &=\sum\limits_{n=0}^{\infty }\frac{\left( \mu \right) _{n}\left(
\nu \right) _{n}}{\left( c\right) _{n}n!}\left( \theta ^{\alpha }+\mu
\right) x^{\alpha n} \\
&=\sum\limits_{n=0}^{\infty }\frac{\left( \mu \right) _{n}\left( \nu
\right) _{n}}{\left( c\right) _{n}n!}\left( n+\mu \right) x^{\alpha
n}=\sum\limits_{n=0}^{\infty }\frac{\left( \mu \right) _{n+1}\left( \nu
\right) _{n}}{\left( c\right) _{n}n!}x^{\alpha n} \\
&=\sum\limits_{n=0}^{\infty }\frac{\mu \left( \mu +1\right) _{n}\left( \nu
\right) _{n}}{\left( c\right) _{n}n!}x^{\alpha n}=\mu _{2}F_{1}\left( \mu +1,\nu
;c;x^{\alpha }\right).
\end{split}
\end{equation*}

Similarly, we have
\begin{equation*}
\left( \theta ^{\alpha }+\nu \right) _{2}F_{1}\left( \mu ,\nu ;c;x^{\alpha
}\right) =\nu _{2}F_{1}\left( \mu ,\nu +1;c;x^{\alpha }\right).
\end{equation*}

Along the same way
\begin{equation*}
\left( \theta ^{\alpha }+c-1\right) _{2}F_{1}\left( \mu ,\nu ;c;x^{\alpha
}\right) =\sum\limits_{n=0}^{\infty }\frac{\left( \mu \right) _{n}\left( \nu
\right) _{n}}{\left( c\right) _{n}n!}\left( \theta ^{\alpha }+c-1\right)
x^{\alpha n}=\sum\limits_{n=0}^{\infty }\frac{\left( \mu \right) _{n}\left(
\nu \right) _{n}}{\left( c\right) _{n}n!}\left( n+c-1\right) x^{\alpha n}.
\end{equation*}

Therefore,
\begin{equation*}
\begin{split}
\left( \theta ^{\alpha }+c-1\right) _{2}F_{1}\left( \mu ,\nu ;c;x^{\alpha
}\right) &=\sum\limits_{n=0}^{\infty }\frac{\left( \mu \right) _{n}\left(
\nu \right) _{n}}{\left( c\right) _{n-1}n!}x^{\alpha n} \\
&=\sum\limits_{n=0}^{\infty }\frac{\left( c-1\right) \left( \mu \right)
_{n}\left( \nu \right) _{n}}{\left( c-1\right) _{n}n!}x^{\alpha n}=\left(
c-1\right) _{2}F_{1}\left( \mu ,\nu ;c-1;x^{\alpha }\right).
\end{split}
\end{equation*}
\end{proof}

The following result is a consequence of Lemma \ref{Lem6.1}.

\begin{thm}\label{Thm6.1}
Let $\alpha \in (0,1],$ then the CFGHF $_{2}F_{1}\left( \mu ,\nu ;c;x^{\alpha }\right) $ satisfies the following contiguous relations 
\begin{equation}\label{eq6.5}
\left( \mu -\nu \right) _{2}F_{1}\left( \mu ,\nu ;c;x^{\alpha }\right) =\mu
_{2}F_{1}\left( \mu +1,\nu ;c;x^{\alpha }\right) -\nu _{2}F_{1}\left( \mu
,\nu +1;c;x^{\alpha }\right)
\end{equation}
and
\begin{equation}\label{eq6.6}
\left( \mu +c-1\right) _{2}F_{1}\left( \mu ,\nu ;c;x^{\alpha }\right) =\mu
_{2}F_{1}\left( \mu +1,\nu ;c;x^{\alpha }\right) -\left( c-1\right)
_{2}F_{1}\left( \mu ,\nu ;c-1;x^{\alpha }\right)
\end{equation}
\end{thm}

\begin{proof}
Using \eqref{eq6.2} and \eqref{eq6.3} immediately give \eqref{eq6.5} and similarly \eqref{eq6.2} and \eqref{eq6.4} assert \eqref{eq6.6}.
\end{proof}

\begin{thm}\label{Thm6.2}
Let $\alpha \in (0,1],$ then the CFGHF, $_{2}F_{1}\left( \mu ,\nu ;c;x^{\alpha }\right) $ satisfies the following contiguous relation
\begin{equation}\label{eq6.7}
\begin{aligned}
\left[ \mu +\left( \nu -c\right) x^{\alpha }\right] ~_{2}F_{1}\left( \mu
,\nu ;c;x^{\alpha }\right) =&\mu \left( 1-x^{\alpha }\right)
~_{2}F_{1}\left( \mu +1,\nu ;c;x^{\alpha }\right) \\
\\
&-c^{-1}\left( c-\mu \right) \left( c-\nu \right) x^{\alpha
}~_{2}F_{1}\left( \mu ,\nu ;c+1;x^{\alpha }\right)
\end{aligned}
\end{equation}
\end{thm}

\begin{proof}
Consider
\begin{equation*}
\theta ^{\alpha }{}_{2}F_{1}\left( \mu ,\nu ;c;x^{\alpha }\right) =\sum\limits_{n=1}^{\infty }
\frac{\left( \mu \right) _{n}\left( \nu \right) _{n}}{\left( c\right) _{n}n!}
nx^{\alpha n}
\end{equation*}

A shift of index gives
\begin{equation}\label{eq6.8}
\theta ^{\alpha }{}_{2}F_{1}\left( \mu ,\nu ;c;x^{\alpha }\right)
=\sum\limits_{n=0}^{\infty }\frac{\left( \mu \right) _{n+1}\left( \nu
\right) _{n+1}}{\left( c\right) _{n+1}n!}x^{\alpha \left( n+1\right)
}=x^{\alpha }\sum\limits_{n=0}^{\infty }\frac{\left( \mu +n\right) \left(
\nu +n\right) }{\left( c+n\right) }\frac{\left( \mu \right) _{n}\left( \nu
\right) _{n}}{\left( c\right) _{n}n!}x^{\alpha n}
\end{equation}

Since, $$\frac{\left( \mu +n\right) \left( \nu +n\right) }{\left(
c+n\right) }=n+\left( \mu +\nu -c\right) +\frac{\left( c-\mu \right) \left(
c-\nu \right) }{c+n},$$
then equation \eqref{eq6.8} yields
\begin{equation*}
\begin{split}
\theta ^{\alpha }{}_{2}F_{1}\left( \mu ,\nu ;c;x^{\alpha }\right)
=&x^{\alpha }\sum\limits_{n=0}^{\infty }\frac{\left( \mu \right) _{n}\left(
\nu \right) _{n}}{\left( c\right) _{n}n!}nx^{\alpha n}+\left( \mu +\nu
-c\right) x^{\alpha }\sum\limits_{n=0}^{\infty }\frac{\left( \mu \right)
_{n}\left( \nu \right) _{n}}{\left( c\right) _{n}n!}x^{\alpha n} \\
&+x^{\alpha }\frac{\left( c-\mu \right) \left( c-\nu \right) }{c}
\sum\limits_{n=0}^{\infty }\frac{c}{c+n}\frac{\left( \mu \right) _{n}\left(
\nu \right) _{n}}{\left( c\right) _{n}n!}x^{\alpha n} \\
=&x^{\alpha }\theta ^{\alpha }{}_{2}F_{1}\left( \mu ,\nu ;c;x^{\alpha
}\right) +\left( \mu +\nu -c\right) x^{\alpha }~_{2}F_{1}\left( \mu ,\nu
;c;x^{\alpha }\right) \\
&+\frac{\left( c-\mu \right) \left( c-\nu \right) }{c}x^{\alpha
}\sum\limits_{n=0}^{\infty }\frac{\left( \mu \right) _{n}\left( \nu \right)
_{n}}{\left( c+1\right) _{n}n!}x^{\alpha n}
\end{split}
\end{equation*}

Hence we can write 
\begin{equation}\label{eq6.9}
\begin{aligned}
\left( 1-x^{\alpha }\right) \theta ^{\alpha }{}_{2}F_{1}\left( \mu ,\nu ;c;x^{\alpha }\right)=&\left( \mu +\nu -c\right) x^{\alpha }~_{2}F_{1}\left( \mu ,\nu
;c;x^{\alpha }\right) \\
&+c^{-1}\left( c-\mu \right) \left( c-\nu \right) x^{\alpha
}~_{2}F_{1}\left( \mu ,\nu ;c+1;x^{\alpha }\right)
\end{aligned}
\end{equation}

From \eqref{eq6.2} we obtain
\begin{equation*}
\begin{split}
\left( 1-x^{\alpha }\right) \theta ^{\alpha }{}_{2}F_{1}\left( \mu ,\nu
;c;x^{\alpha }\right) =&-\mu \left( 1-x^{\alpha }\right) ~_{2}F_{1}\left(
\mu ,\nu ;c;x^{\alpha }\right) \\
&+\mu \left( 1-x^{\alpha }\right) ~_{2}F_{1}\left( \mu +1,\nu ;c;x^{\alpha
}\right)
\end{split}
\end{equation*}
which implies together with \eqref{eq6.9} the required relation.
\end{proof}

\begin{thm}\label{Thm6.3}
For $\alpha \in (0,1],$ then the CFGHF, $_{2}F_{1}\left( \mu ,\nu ;c;x^{\alpha }\right) $ satisfies the following contiguous relation
\begin{equation}\label{eq6.10}
\left( 1-x^{\alpha }\right) ~_{2}F_{1}\left( \mu ,\nu ;c;x^{\alpha }\right)
=_{2}F_{1}\left( \mu -1,\nu ;c;x^{\alpha }\right) -c^{-1}\left( c-\nu
\right) x^{\alpha }~~_{2}F_{1}\left( \mu ,\nu ;c+1;x^{\alpha }\right). 
\end{equation}
\begin{equation}\label{eq6.11}
\left( 1-x^{\alpha }\right) ~_{2}F_{1}\left( \mu ,\nu ;c;x^{\alpha }\right)
=~_{2}F_{1}\left( \mu ,\nu -1;c;x^{\alpha }\right) -c^{-1}\left( c-\mu
\right) x^{\alpha }~~_{2}F_{1}\left( \mu ,\nu ;c+1;x^{\alpha }\right) 
\end{equation}
\end{thm}
\begin{proof}
By operating  $\theta ^{\alpha }{}$  $
_{2}F_{1}\left( \mu -1,\nu ;c;x^{\alpha }\right) ,$  we  obtain 
\begin{equation*}
\theta ^{\alpha }~_{2}F_{1}\left( \mu -1,\nu ;c;x^{\alpha }\right) =\theta ^{\alpha }~\sum\limits_{n=0}^{\infty }\frac{\left( \mu -1\right) _{n}\left(
\nu \right) _{n}}{\left( c\right) _{n}n!}x^{\alpha
n}=\sum\limits_{n=1}^{\infty }\frac{\left( \mu -1\right) _{n}\left(
\nu \right) _{n}}{\left( c\right) _{n}n!}nx^{\alpha n}.
\end{equation*}

A shift of index yields 
\begin{equation}\label{eq6.12}
\begin{aligned}
\theta ^{\alpha }~_{2}F_{1}\left( \mu -1,\nu ;c;x^{\alpha }\right) 
&=\sum\limits_{n=0}^{\infty }\frac{\left( \mu -1\right) _{n+1}\left( \nu
\right) _{n+1}}{\left( c\right) _{n+1}n!}x^{\alpha \left( n+1\right) } \\
&=\left( \mu -1\right) x^{\alpha }\sum\limits_{n=0}^{\infty }\frac{\left(
\nu +n\right) }{\left( c+n\right) }\frac{\left( \mu \right) _{n}\left( \nu
\right) _{n}}{\left( c\right) _{n}n!}x^{\alpha n}
\end{aligned}
\end{equation}

But $\frac{\left( \nu +n\right) }{\left( c+n\right) }=1-\frac{c-\nu }{c+n},$
thus \eqref{eq6.12} becomes 
\begin{equation*}
\begin{split}
\theta ^{\alpha }~_{2}F_{1}\left( \mu -1,\nu ;c;x^{\alpha }\right) 
&=\left( \mu -1\right) x^{\alpha }\left[ \sum\limits_{n=0}^{\infty }\frac{%
\left( \mu \right) _{n}\left( \nu \right) _{n}}{\left( c\right) _{n}n!}%
x^{\alpha n}+\frac{c-\nu }{c}\sum\limits_{n=0}^{\infty }\frac{c}{c+n}\frac{%
\left( \mu \right) _{n}\left( \nu \right) _{n}}{\left( c\right) _{n}n!}%
x^{\alpha n}\right]  \\
&=\left( \mu -1\right) x^{\alpha }\left[ _{2}F_{1}\left( \mu ,\nu
;c;x^{\alpha }\right) -\frac{c-\nu }{c}~_{2}F_{1}\left( \mu ,\nu
;c+1;x^{\alpha }\right) \right] 
\end{split}
\end{equation*}
which yield 
\begin{equation}\label{eq6.13}
\begin{aligned}
\theta ^{\alpha }~_{2}F_{1}\left( \mu -1,\nu ;c;x^{\alpha }\right) 
&=\left( \mu -1\right) x^{\alpha }~_{2}F_{1}\left( \mu ,\nu ;c;x^{\alpha
}\right)  \\
&-c^{-1}\left( c-\nu \right) \left( \mu -1\right) x^{\alpha
}~~_{2}F_{1}\left( \mu ,\nu ;c+1;x^{\alpha }\right).
\end{aligned} 
\end{equation}

Now,  replacing $\mu $ by $\left( \mu -1\right) $ in \eqref{eq6.2} implies that 
\begin{equation}\label{eq6.14}
\theta ^{\alpha }~_{2}F_{1}\left( \mu -1,\nu ;c;x^{\alpha }\right) =-\left(
\mu -1\right) ~_{2}F_{1}\left( \mu -1,\nu ;c;x^{\alpha }\right) +\left( \mu
-1\right) ~_{2}F_{1}\left( \mu ,\nu ;c;x^{\alpha }\right). 
\end{equation}

From \eqref{eq6.13} and \eqref{eq6.14}, the relation  \eqref{eq6.10} is verified. 
Similarly, since $\mu $ and $\nu $ can be interchanged without affecting the
hypergeometric series,  \eqref{eq6.11} yields. 
\end{proof}

Observe that from the contiguous relations we just derived in Theorems \ref{Thm6.1}, \ref{Thm6.2}, and  \ref{Thm6.3}, we can obtain further relations by performing some suitable eliminations as follows.

From \eqref{eq6.7} and \eqref{eq6.10}, we get  
\begin{equation}\label{eq6.15}
\begin{aligned}
\left[ 2\mu -c+\left( \nu -\mu \right) x^{\alpha }\right] ~_{2}F_{1}\left(
\mu ,\nu ;c;x^{\alpha }\right)  =&\mu \left( 1-x^{\alpha }\right)
~_{2}F_{1}\left( \mu +1,\nu ;c;x^{\alpha }\right)  \\
&-\left( c-\mu \right) ~_{2}F_{1}\left( \mu -1,\nu ;c;x^{\alpha }\right).
\end{aligned}
\end{equation}

A combination of \eqref{eq6.7} and \eqref{eq6.11} gives 
\begin{equation}\label{eq6.16}
\begin{aligned}
\left[ \mu +\gamma -c\right] ~_{2}F_{1}\left( \mu ,\nu ;c;x^{\alpha }\right)
=&\mu \left( 1-x^{\alpha }\right) ~_{2}F_{1}\left( \mu +1,\nu ;c;x^{\alpha
}\right)  \\
&-\left( c-\gamma \right) ~_{2}F_{1}\left( \mu ,\nu -1;c;x^{\alpha }\right).
\end{aligned}
\end{equation}

Inserting \eqref{eq6.5} in \eqref{eq6.15}, satisfies
\begin{equation}\label{eq6.17}
\begin{aligned}
\left[ c-\mu -\nu \right] _{2}F_{1}\left( \mu ,\nu ;c;x^{\alpha }\right) 
=&\left( c-\mu \right) ~_{2}F_{1}\left( \mu -1,\nu ;c;x^{\alpha }\right)  \\
&-\nu \left( 1-x^{\alpha }\right) _{2}F_{1}\left( \mu ,\nu +1;c;x^{\alpha
}\right).
\end{aligned}
\end{equation}

Also, from \eqref{eq6.15} and \eqref{eq6.16}, we get 
\begin{equation}\label{eq6.18}
\left( \nu -\mu \right) \left( 1-x^{\alpha }\right) _{2}F_{1}\left( \mu ,\nu
;c;x^{\alpha }\right) =\left( c-\mu \right) ~_{2}F_{1}\left( \mu -1,\nu
;c;x^{\alpha }\right) -\left( c-\nu \right) ~_{2}F_{1}\left( \mu ,\nu
-1;c;x^{\alpha }\right).
\end{equation}

Use \eqref{eq6.6} and \eqref{eq6.16} to obtain  
\begin{equation}\label{eq6.19}
\begin{aligned}
\left[ 1-\mu +\left( c-\nu -1\right) x^{\alpha }\right] ~_{2}F_{1}\left( \mu
,\nu ;c;x^{\alpha }\right)  =&\left( c-\mu \right) _{2}F_{1}\left( \mu
-1,\nu ;c;x^{\alpha }\right)  \\
&-\left( c-1\right) \left( 1-x^{\alpha }\right) _{2}F_{1}\left( \mu ,\nu
;c-1;x^{\alpha }\right).
\end{aligned}
\end{equation}

By interchanging $\mu $ and $\nu $ in \eqref{eq6.15},  we have 
\begin{equation}\label{eq6.20}
\begin{aligned}
\left[ 2\nu -c+\left( \mu -\nu \right) x^{\alpha }\right] ~_{2}F_{1}\left(
\mu ,\nu ;c;x^{\alpha }\right)  =&\nu \left( 1-x^{\alpha }\right)
~_{2}F_{1}\left( \mu ,\nu +1;c;x^{\alpha }\right)  \\
&-\left( c-\nu \right) ~_{2}F_{1}\left( \mu ,\nu -1;c;x^{\alpha }\right).
\end{aligned}
\end{equation}

We append this section by driving the CFGHE.
The conformable fractional operator \eqref{eq6.1} can be employed to derive a conformable fractional differential equation characterized by \eqref{eq3.3}.
 
Relation \eqref{eq3.3} with the operator  $\theta ^{\alpha }$ defined by \eqref{eq6.1} gives
\begin{equation*}
\theta ^{\alpha }\left( \theta ^{\alpha }+c-1\right) y=\theta ^{\alpha
}\sum\limits_{n=0}^{\infty }\frac{\left( \mu \right) _{n}\left( \nu \right)
_{n}}{\left( c\right) _{n}n!}\left( n+c-1\right) x^{\alpha
n}=\sum\limits_{n=1}^{\infty }\frac{\left( \mu \right) _{n}\left( \nu
\right) _{n}}{\left( c\right) _{n}n!}n\left( n+c-1\right) x^{\alpha n}
\end{equation*}%

A shift of index yields 
\begin{eqnarray*}
\theta ^{\alpha }\left( \theta ^{\alpha }+c-1\right) y
&=&\sum\limits_{n=0}^{\infty }\frac{\left( \mu \right) _{n+1}\left( \nu
\right) _{n+1}}{\left( c\right) _{n+1}n!}\left( n+c\right) x^{\alpha \left(
n+1\right) } \\
&=&x^{\alpha }\sum\limits_{n=0}^{\infty }\frac{\left( \mu \right) _{n}\left(
\nu \right) _{n}}{\left( c\right) _{n}n!}\left( n+\mu \right) \left( n+\nu
\right) x^{\alpha n}=x^{\alpha }\left( \theta ^{\alpha }+\mu \right) \left(
\theta ^{\alpha }+\nu \right) y
\end{eqnarray*}%

This shows $y=~_{2}F_{1}\left( \mu ,\nu ;c;x^{\alpha }\right) $ is a
solution of the following  CFDE
\begin{equation}\label{eq7.1}
\left[ \theta ^{\alpha }\left( \theta ^{\alpha }+c-1\right) -x^{\alpha
}\left( \theta ^{\alpha }+\mu \right) \left( \theta ^{\alpha }+\nu \right) %
\right] y=0,~\ \ \ \ \ \ \theta ^{\alpha }=\frac{1}{\alpha }x^{\alpha
}D^{\alpha }
\end{equation}

Owing to $\theta ^{\alpha }y=\frac{1}{\alpha }x^{\alpha }D^{\alpha
}y$ and $\theta ^{\alpha }\theta ^{\alpha }y=\frac{1}{\alpha ^{2}}x^{2\alpha
}D^{\alpha }D^{\alpha }y+\frac{1}{\alpha }x^{\alpha }D^{\alpha },$ then equation \eqref{eq7.1} can be written in the form 
\begin{equation*}
x^{\alpha }\left( 1-x^{\alpha }\right) D^{\alpha }D^{\alpha }y+\alpha \left[
c-\left( \mu +\nu +1\right) x^{\alpha }\right] D^{\alpha }y-\alpha ^{2}\mu
\nu y=0,
\end{equation*}
 which coincide with \eqref{eq3.1}
\section{Conformable fractional integral of CFGHF}
Taking into account the $\alpha$-integral given in  Definition \ref{DefFI}, we provide some forms of fractional integral related to the 
$\alpha $-Gauss
hypergeometric function.
Thus according to Definition \ref{DefFI}, it follows that
\begin{equation}\label{eq8.1}
I_{\alpha }f\left( x\right) =\int\limits_{0}^{x}t^{\alpha -1}f\left(t\right) dt. 
\end{equation}
 
In this regard, we state the following important result  given in \cite{khalil2014new}.
\begin{lem}
 Suppose that $f:\left[ 0,\infty \right) \rightarrow 
\mathbb{R}$  is $\alpha $-differentiable for $ \alpha \in (0,1],$ then
for all $ x>0$ one can write: 
\begin{equation}\label{eq8.2}
I_{\alpha }D^{\alpha }\left( f\left( x\right) \right) =f\left( x\right)
-f\left( 0\right) 
\end{equation}
\end{lem}

With the aid of \eqref{eq8.1} and  \eqref{eq8.2}, the following result can be
deduced.
\begin{thm}\label{thm7.1}
For $\alpha \in (0,1]$, then the conformable fractional integral $
I_{\alpha }$ of CFGHF, $_{2}F_{1}\left(
\mu ,\nu ;c;x^{\alpha }\right) $ can be written as
\begin{equation}\label{eq8.3}
I_{\alpha }~_{2}F_{1}\left( \mu ,\nu ;c;x^{\alpha }\right) =\frac{\left(
c-1\right) }{\alpha \left( \mu -1\right) \left( \nu -1\right) }\left[
_{2}F_{1}\left( \mu -1,\nu -1;c-1;x^{\alpha }\right) -1\right] 
\end{equation}
\end{thm}
\begin{proof}
Relation \eqref{eq5.9}, gives
\begin{equation*}
D^{\alpha }~_{2}F_{1}\left( \mu -1,\nu -1;c-1;x^{\alpha }\right) =\frac{
\alpha \left( \mu -1\right) \left( \nu -1\right) }{\left( c-1\right) }
~_{2}F_{1}\left( \mu ,\nu ;c;x^{\alpha }\right) 
\end{equation*}

Acting by the conformable fractional integral on both sides we obtain
\begin{equation*}
I_{\alpha }D^{\alpha }~ _{2}F_{1}\left( \mu -1,\nu -1;c-1;x^{\alpha }\right) =
\frac{\alpha \left( \mu -1\right) \left( \nu -1\right) }{\left( c-1\right) }
 I_{\alpha }~_{2}F_{1}\left( \mu ,\nu ;c;x^{\alpha }\right) 
\end{equation*}

Using \eqref{eq8.2}, we have 
\begin{equation*}
_{2}F_{1}\left( \mu -1,\nu -1;c-1;x^{\alpha }\right) -1=\frac{\alpha \left(
\mu -1\right) \left( \nu -1\right) }{\left( c-1\right) } I_{\alpha
}~_{2}F_{1}\left( \mu ,\nu ;c;x^{\alpha }\right) 
\end{equation*}
Therefore 
\begin{equation*}
I_{\alpha }~_{2}F_{1}\left( \mu ,\nu ;c;x^{\alpha }\right) =\frac{\left(
c-1\right) }{\alpha \left( \mu -1\right) \left( \nu -1\right) }\left[
_{2}F_{1}\left( \mu -1,\nu -1;c-1;x^{\alpha }\right) -1\right] 
\end{equation*}
as required.
\end{proof}
\begin{thm}\label{thm7.2}
For $ \alpha \in (0,1]$, then the CFGHF, $_{2}F_{1}\left( \mu ,\nu ;c;x^{\alpha }\right) $ has a
conformable fractional integral representation in the form 
\begin{equation*}
_{2}F_{1}\left( \mu ,\nu ;c;x^{\alpha }\right) =1+\frac{\alpha \mu \nu }{c}
\int\limits_{0}^{x}~_{2}F_{1}\left( \mu +1,\nu +1;c+1;t^{\alpha }\right)
d_{\alpha }t
\end{equation*}
where $d_{\alpha }t=t^{\alpha -1}dt$
\end{thm}
\begin{proof}
In view of theorem \ref{thm7.1}, we obtain 
\begin{equation*}
I_{\alpha }~_{2}F_{1}\left( \mu +1,\nu +1;c+1;x^{\alpha }\right) =
\frac{c}{\alpha \mu \nu }\left[ _{2}F_{1}\left( \mu ,\nu ;c;x^{\alpha
}\right) -1\right] 
\end{equation*}

Hence, 
\begin{eqnarray*}
_{2}F_{1}\left( \mu ,\nu ;c;x^{\alpha }\right)  &=&1+\frac{\alpha \mu \nu }{c}
I_{\alpha }\left[ _{2}F_{1}\left( \mu +1,\nu +1;c+1;x^{\alpha }\right) 
\right]  \\
&=&1+\frac{\alpha \mu \nu }{c}\int\limits_{0}^{x}~_{2}F_{1}\left( \mu +1,\nu
+1;c+1;t^{\alpha }\right) d_{\alpha }t \\
&=&1+\frac{\alpha \mu \nu }{c}\int\limits_{0}^{x}~_{2}F_{1}\left( \mu +1,\nu
+1;c+1;t^{\alpha }\right) ~t^{\alpha -1}dt
\end{eqnarray*}
as required.
\end{proof}
Now, following \cite{rainville1969special}, we state the following result
\begin{thm}\label{integral representation}
For $\alpha \in (0,1]$ and $c>\nu >0,$ the CFGHF, $_{2}F_{1}\left( \mu ,\nu
;c;x^{\alpha }\right) $ has an integral representation
\begin{equation}\label{eq8.4}
_{2}F_{1}\left( \mu ,\nu ;c;x^{\alpha }\right) =\frac{\Gamma \left( c\right) 
}{\Gamma \left( \nu \right) \Gamma \left( c-\nu \right) }\int\limits_{0}^{1}%
\tau ^{\nu -1}\left( 1-\tau \right) ^{c-\nu -1}\left( 1-x^{\alpha }\tau
\right) ^{-\mu }d\tau
\end{equation}
\end{thm}

\begin{proof}
From the definition of CFGHF \eqref{eq3.3}, we have
\begin{equation*}
\begin{split}
_{2}F_{1}\left( \mu ,\nu ;c;x^{\alpha }\right) &=\sum\limits_{n=0}^{\infty }\frac{\left( \mu \right) _{n}~\Gamma
\left( \nu +n\right) ~\Gamma \left( c\right) }{\Gamma \left( \nu \right)
\Gamma \left( c+n\right) n!}x^{\alpha n} \\
&=\frac{\Gamma \left( c\right) }{\Gamma \left( \nu \right) \Gamma \left(
c-\nu \right) }\sum\limits_{n=0}^{\infty }\frac{\left( \mu \right)
_{n}~\Gamma \left( \nu +n\right) ~\Gamma \left( c-\nu \right) }{\Gamma
\left( c+n\right) n!}x^{\alpha n} \\
&=\frac{\Gamma \left( c\right) }{\Gamma \left( \nu \right) \Gamma \left(
c-\nu \right) }\sum\limits_{n=0}^{\infty }\beta \left( c-\nu ,\nu +n\right) 
\frac{\left( \mu \right) _{n}~}{n!}x^{\alpha n}.
\end{split}
\end{equation*}

Using the integral form of beta function, we get
\begin{equation*}
_{2}F_{1}\left( \mu ,\nu ;c;x^{\alpha }\right)  =\frac{\Gamma \left( c\right) }{\Gamma \left( \nu \right) \Gamma \left(
c-\nu \right) }\int\limits_{0}^{1}\tau ^{\nu -1}\left( 1-\tau \right)
^{c-\nu -1}.\sum\limits_{n=0}^{\infty }\frac{\left( \mu \right) _{n}~}{n!}%
\left( x^{\alpha }\tau \right) ^{n}d\tau. 
\end{equation*}

By using the identity $\sum\limits_{n=0}^{\infty }\frac{\left( \mu \right) _{n}%
}{n!}~t^{n}=\left( 1-t\right) ^{-\left( \mu \right) };~~\left\vert
t\right\vert <1$, we have 
\begin{equation*}
_{2}F_{1}\left( \mu ,\nu ;c;x^{\alpha }\right) =\frac{\Gamma \left( c\right) 
}{\Gamma \left( \nu \right) \Gamma \left( c-\nu \right) }\int\limits_{0}^{1}%
\tau ^{\nu -1}\left( 1-\tau \right) ^{c-\nu -1}.\left( 1-x^{\alpha }\tau
\right) ^{-\mu }d\tau
\end{equation*}
as required.
\end{proof}
\section{Recursion formulas for $_{2}F_{1}\left( \protect\mu ,\nu ;c;x^{\protect\alpha }\right) $}
Employing the assertion in theorem \ref{integral representation}, and owing to the results given in \cite{opps2009recursion} we state the following recursion  formulas.
\begin{thm}\label{tR1}
The following recursion formulas hold  for the CFGHF 
\begin{align}
_{2}F_{1}\left( \mu +n,\nu ;c;x^{\alpha }\right) &=~_{2}F_{1}\left( \mu
,\nu ;c;x^{\alpha }\right) +\frac{\nu x^{\alpha }}{c}\sum
\limits_{k=1}^{n}~_{2}F_{1}\left( \mu +n-k+1,\nu +1;c+1;x^{\alpha }\right) \label{eq9.1}\\
_{2}F_{1}\left( \mu -n,\nu ;c;x^{\alpha }\right) &=~_{2}F_{1}\left( \mu
,\nu ;c;x^{\alpha }\right) -\frac{\nu x^{\alpha }}{c}\sum%
\limits_{k=1}^{n}~_{2}F_{1}\left( \mu -k+1,\nu +1;c+1;x^{\alpha }\right) , \label{eq9.2}
\end{align} 
where $\left\vert x^{\alpha }\right\vert <1,~n\in 
\mathbb{N}
_{0}=
\mathbb{N}
\cup \left\{ 0\right\}.$
\end{thm}
\begin{proof}
By means of \eqref{eq8.4}, we have
\begin{equation*}
\begin{split}
_{2}F_{1}\left( \mu +n,\nu ;c;x^{\alpha }\right) =&\frac{\Gamma \left( c\right) }{\Gamma \left( \nu \right) \Gamma \left(
c-\nu \right) }\int\limits_{0}^{1}\tau ^{\nu -1}\left( 1-\tau \right)
^{c-\nu -1}\left( 1-x^{\alpha }\tau \right) ^{-\mu -n-1}d\tau \\
&-\frac{\Gamma \left( c\right) x^{\alpha }}{\Gamma \left( \nu \right)
\Gamma \left( c-\nu \right) }\int\limits_{0}^{1}\tau ^{\nu }\left( 1-\tau
\right) ^{c-\nu -1}\left( 1-x^{\alpha }\tau \right) ^{-\mu -n-1}d\tau \\
=&\frac{\Gamma \left( c\right) }{\Gamma \left( \nu \right) \Gamma \left(
c-\nu \right) }\int\limits_{0}^{1}\tau ^{\nu -1}\left( 1-\tau \right)
^{c-\nu -1}\left( 1-x^{\alpha }\tau \right) ^{-\mu -n-1}d\tau \\
&-\frac{x^{\alpha }}{\alpha \left( \mu +n\right) }\left\{ \frac{\alpha
\left( \mu +n\right) \Gamma \left( c\right) }{\Gamma \left( \nu \right)
\Gamma \left( c-\nu \right) }\int\limits_{0}^{1}\tau ^{\nu }\left( 1-\tau
\right) ^{c-\nu -1}\left( 1-x^{\alpha }\tau \right) ^{-\mu -n-1}d\tau
\right\}
\end{split}
\end{equation*}
In virtue of conformable derivative, we may write
\begin{align*}
_{2}F_{1}\left( \mu +n,\nu ;c;x^{\alpha }\right) &=\frac{\Gamma \left(
c\right) }{\Gamma \left( \nu \right) \Gamma \left( c-\nu \right) }%
\int\limits_{0}^{1}\tau ^{\nu -1}\left( 1-\tau \right) ^{c-\nu -1}\left(
1-x^{\alpha }\tau \right) ^{-\mu -n-1}d\tau \\
&-\frac{x^{\alpha }}{\alpha \left( \mu +n\right) }D^{\alpha }\left\{ \frac{%
\Gamma \left( c\right) }{\Gamma \left( \nu \right) \Gamma \left( c-\nu
\right) }\int\limits_{0}^{1}\tau ^{\nu -1}\left( 1-\tau \right) ^{c-\nu
-1}\left( 1-x^{\alpha }\tau \right) ^{-\mu -n}d\tau \right\}.
\end{align*}

Again using \eqref{eq8.4}, we have
\begin{equation*}
_{2}F_{1}\left( \mu +n,\nu ;c;x^{\alpha }\right) =~_{2}F_{1}\left( \mu
+n+1,\nu ;c;x^{\alpha }\right) -\frac{x^{\alpha }}{\alpha \left( \mu
+n\right) }D^{\alpha }\left\{ _{2}F_{1}\left( \mu +n,\nu ;c;x^{\alpha
}\right) \right\}.
\end{equation*}

Thus,
\begin{equation*}
_{2}F_{1}\left( \mu +n-1,\nu ;c;x^{\alpha }\right) =~_{2}F_{1}\left( \mu
+n,\nu ;c;x^{\alpha }\right) -\frac{x^{\alpha }}{\alpha \left( \mu
+n-1\right) }D^{\alpha }\left\{ _{2}F_{1}\left( \mu +n-1,\nu ;c;x^{\alpha
}\right) \right\} ,
\end{equation*}
or
\begin{equation}\label{eq9.3}
_{2}F_{1}\left( \mu +n,\nu ;c;x^{\alpha }\right) =~_{2}F_{1}\left( \mu
+n-1,\nu ;c;x^{\alpha }\right) ~+\frac{x^{\alpha }}{\alpha \left( \mu
+n-1\right) }D^{\alpha }\left\{ _{2}F_{1}\left( \mu +n-1,\nu ;c;x^{\alpha
}\right) \right\}
\end{equation}

Applying this last identity \eqref{eq9.3}, we get
\begin{eqnarray*}
_{2}F_{1}\left( \mu +n,\nu ;c;x^{\alpha }\right) &=&~_{2}F_{1}\left( \mu
+n-2,\nu ;c;x^{\alpha }\right) +\frac{x^{\alpha }}{\alpha \left( \mu
+n-2\right) }D^{\alpha }\left\{ _{2}F_{1}\left( \mu +n-2,\nu ;c;x^{\alpha
}\right) \right\} \\
&&~+\frac{x^{\alpha }}{\alpha \left( \mu +n-1\right) }D^{\alpha }\left\{
_{2}F_{1}\left( \mu +n-1,\nu ;c;x^{\alpha }\right) \right\} \\
&=&~_{2}F_{1}\left( \mu +n-2,\nu ;c;x^{\alpha }\right) +\frac{x^{\alpha }}{%
\alpha }.\sum\limits_{k=1}^{2}\frac{1}{\left( \mu +n-k\right) }D^{\alpha
}\left\{ _{2}F_{1}\left( \mu +n-k,\nu ;c;x^{\alpha }\right) \right\} .
\end{eqnarray*}

Again apply \eqref{eq9.3} recursively n-times, we obtain
\begin{equation}\label{eq9.4}
_{2}F_{1}\left( \mu +n,\nu ;c;x^{\alpha }\right) =~_{2}F_{1}\left( \mu ,\nu
;c;x^{\alpha }\right) +\frac{x^{\alpha }}{\alpha }.\sum\limits_{k=1}^{n}
\frac{1}{\left( \mu +n-k\right) }D^{\alpha }\left\{ _{2}F_{1}\left( \mu
+n-k,\nu ;c;x^{\alpha }\right) \right\} .
\end{equation}

Using \eqref{eq5.9}, we have
\begin{equation*}
_{2}F_{1}\left( \mu +n,\nu ;c;x^{\alpha }\right) =~_{2}F_{1}\left( \mu ,\nu
;c;x^{\alpha }\right) +\frac{x^{\alpha }\nu }{c}.\sum\limits_{k=1}^{n}\left%
\{ _{2}F_{1}\left( \mu +n-k+1,\nu +1;c+1;x^{\alpha }\right) \right\} .
\end{equation*}

Furthermore,  the assertion of theorem \ref{integral representation} gives 
\begin{equation*}
\begin{split}
_{2}F_{1}\left( \mu -n,\nu ;c;x^{\alpha }\right) =&\frac{\Gamma \left( c\right) }{\Gamma \left( \nu \right) \Gamma \left(
c-\nu \right) }\int\limits_{0}^{1}\tau ^{\nu -1}\left( 1-\tau \right)
^{c-\nu -1}\left( 1-x^{\alpha }\tau \right) ^{n-\mu -1}d\tau \\
&-\frac{\Gamma \left( c\right) x^{\alpha }}{\Gamma \left( \nu \right)
\Gamma \left( c-\nu \right) }\int\limits_{0}^{1}\tau ^{\nu }\left( 1-\tau
\right) ^{c-\nu -1}\left( 1-x^{\alpha }\tau \right) ^{n-\mu -1}d\tau \\
=&\frac{\Gamma \left( c\right) }{\Gamma \left( \nu \right) \Gamma \left(
c-\nu \right) }\int\limits_{0}^{1}\tau ^{\nu -1}\left( 1-\tau \right)
^{c-\nu -1}\left( 1-x^{\alpha }\tau \right) ^{n-\mu -1}d\tau \\
&-\frac{x^{\alpha }}{\alpha \left( \mu -n\right) }\left\{ \frac{\alpha
\left( \mu -n\right) \Gamma \left( c\right) }{\Gamma \left( \nu \right)
\Gamma \left( c-\nu \right) }\int\limits_{0}^{1}\tau ^{\nu }\left( 1-\tau
\right) ^{c-\nu -1}\left( 1-x^{\alpha }\tau \right) ^{n-\mu -1}d\tau \right\} \\
=&\frac{\Gamma \left( c\right) }{\Gamma \left( \nu \right) \Gamma \left(
c-\nu \right) }\int\limits_{0}^{1}\tau ^{\nu -1}\left( 1-\tau \right)
^{c-\nu -1}\left( 1-x^{\alpha }\tau \right) ^{n-\mu -1}d\tau \\
&-\frac{x^{\alpha }}{\alpha \left( \mu -n\right) } D^{\alpha } \left\{ \frac{ \Gamma \left( c\right) }{\Gamma \left( \nu \right)
\Gamma \left( c-\nu \right) }\int\limits_{0}^{1}\tau ^{\nu -1 }\left( 1-\tau
\right) ^{c-\nu -1}\left( 1-x^{\alpha }\tau \right) ^{n-\mu }d\tau \right\}
\end{split}
\end{equation*}

Relying on  the integral representation \eqref{eq8.4}, we have
\begin{equation*}
_{2}F_{1}\left( \mu -n,\nu ;c;x^{\alpha }\right) =~_{2}F_{1}\left( \mu
-n+1,\nu ;c;x^{\alpha }\right) -\frac{x^{\alpha }}{\alpha \left( \mu
-n\right) }D^{\alpha }\left\{ _{2}F_{1}\left( \mu -n,\nu ;c;x^{\alpha
}\right) \right\}.
\end{equation*}

Therefore,
\begin{equation*}
_{2}F_{1}\left( \mu -n-1,\nu ;c;x^{\alpha }\right) =~_{2}F_{1}\left( \mu
-n,\nu ;c;x^{\alpha }\right) -\frac{x^{\alpha }}{\alpha \left( \mu
-n-1\right) }D^{\alpha }\left\{ _{2}F_{1}\left( \mu -n-1,\nu ;c;x^{\alpha
}\right) \right\} ,
\end{equation*}
or
\begin{equation}\label{eq9.5}
_{2}F_{1}\left( \mu -n,\nu ;c;x^{\alpha }\right) =~_{2}F_{1}\left( \mu
-n-1,\nu ;c;x^{\alpha }\right) ~+\frac{x^{\alpha }}{\alpha \left( \mu
-n-1\right) }D^{\alpha }\left\{ _{2}F_{1}\left( \mu -n-1,\nu ;c;x^{\alpha
}\right) \right\}.
\end{equation}

Applying relation \eqref{eq9.5} recursively, we obtain
\begin{align*}
_{2}F_{1}\left( \mu -n,\nu ;c;x^{\alpha }\right) &=~_{2}F_{1}\left( \mu
-n-2,\nu ;c;x^{\alpha }\right) +\frac{x^{\alpha }}{\alpha \left( \mu
-n-2\right) }D^{\alpha }\left\{ _{2}F_{1}\left( \mu -n-2,\nu ;c;x^{\alpha
}\right) \right\} \\
&~+\frac{x^{\alpha }}{\alpha \left( \mu -n-1\right) }D^{\alpha }\left\{
_{2}F_{1}\left( \mu -n-1,\nu ;c;x^{\alpha }\right) \right\} \\
&=~_{2}F_{1}\left( \mu -n-2,\nu ;c;x^{\alpha }\right) +\frac{x^{\alpha }}{%
\alpha }\sum\limits_{k=1}^{2}\frac{1}{\left( \mu -n-k\right) }D^{\alpha
}\left\{ _{2}F_{1}\left( \mu -n-k,\nu ;c;x^{\alpha }\right) \right\} .
\end{align*}

Repeating the recurrence relation \eqref{eq9.5} $n$-times and dappling the derivative
formula \eqref{eq5.9}, we have
\begin{equation}\label{eq9.6}
_{2}F_{1}\left( \mu -n,\nu ;c;x^{\alpha }\right) =~_{2}F_{1}\left( \mu
-2n,\nu ;c;x^{\alpha }\right) +\frac{x^{\alpha }\nu }{c}.\sum%
\limits_{k=1}^{n}\left\{ _{2}F_{1}\left( \mu -n-k+1,\nu +1;c+1;x^{\alpha
}\right) \right\} .
\end{equation}

The relation \eqref{eq9.2} follows directly from \eqref{eq9.6} upon replacing $\mu $ by $\left( \mu
+n\right) $ where $ n\in 
\mathbb{N}
_{0}=
\mathbb{N}
\cup \left\{ 0\right\} .$
\end{proof}

\begin{thm}\label{tR2}
For $\alpha \in (0,1],$ the following recursion formulas hold true
for the CFGHF, $_{2}F_{1}\left( \mu ,\nu ;c;x^{\alpha }\right) $ 
\begin{equation}\label{eq9.7}
\begin{split}
_{2}F_{1}\left( \mu +n,\nu ;c;x^{\alpha }\right) 
=&~\sum\limits_{k=0}^{n}\binom{n}{k}
 \frac{\left( \nu \right) _{k}~}{\left( c\right) _{k}}x^{\alpha
k}~_{2}F_{1}\left( \mu +k,\nu +k;c+k;x^{\alpha }\right) ,
\end{split}
\end{equation}
and
\begin{equation}\label{eq9.8}
\begin{split}
_{2}F_{1}\left( \mu -n,\nu ;c;x^{\alpha }\right) 
=&~\sum\limits_{k=0}^{n}\left( -1\right) ^{k}\binom{n}{k} \frac{\left( \nu \right) _{k}~}{\left( c\right) _{k}}x^{\alpha
k}~_{2}F_{1}\left( \mu ,\nu +k;c+k;x^{\alpha }\right) , 
\end{split}
\end{equation}
where 
$\left\vert x^{\alpha }\right\vert <1,~n\in 
\mathbb{N}
_{0}=
\mathbb{N}
\cup \left\{ 0\right\}$.
\end{thm}
\begin{proof}
From \eqref{eq9.1} of theorem \ref{tR1} with $n=1,$ we see
\begin{equation}\label{eq9.9}
_{2}F_{1}\left( \mu +1,\nu ;c;x^{\alpha }\right) =~_{2}F_{1}\left( \mu ,\nu
;c;x^{\alpha }\right) +\frac{\nu x^{\alpha }}{c}~_{2}F_{1}\left( \mu +1,\nu
+1;c+1;x^{\alpha }\right) ,
\end{equation}
with $n=2,$ we have
\begin{equation}\label{eq9.10}
\begin{split}
_{2}F_{1}\left( \mu +2,\nu ;c;x^{\alpha }\right) =&~_{2}F_{1}\left( \mu
,\nu ;c;x^{\alpha }\right) +\frac{\nu x^{\alpha }}{c}~_{2}F_{1}\left( \mu
+1,\nu +1;c+1;x^{\alpha }\right) \\
&+\frac{\nu x^{\alpha }}{c}~_{2}F_{1}\left( \mu +2,\nu +1;c+1;x^{\alpha
}\right)
\end{split}
\end{equation}

Making use of \eqref{eq9.9} and \eqref{eq9.10}, we obtain
\begin{equation}\label{eq9.11}
\begin{split}
_{2}F_{1}\left( \mu +2,\nu ;c;x^{\alpha }\right)  =&~_{2}F_{1}\left( \mu
,\nu ;c;x^{\alpha }\right) +\frac{2\nu x^{\alpha }}{c}~_{2}F_{1}\left( \mu
+1,\nu +1;c+1;x^{\alpha }\right)  \\
&+\frac{\nu \left( \nu +1\right) x^{2\alpha }}{c\left( c+1\right) }%
~_{2}F_{1}\left( \mu +2,\nu +2;c+2;x^{\alpha }\right) .
\end{split}
\end{equation}

Using \eqref{eq9.9} and \eqref{eq9.11} with  $n=3,$ it follows that 
\begin{equation}\label{eq9.12}
\begin{split}
_{2}F_{1}\left( \mu +3,\nu ;c;x^{\alpha }\right)  =&~_{2}F_{1}\left( \mu
,\nu ;c;x^{\alpha }\right) +\frac{3\nu x^{\alpha }}{c}~_{2}F_{1}\left( \mu
+1,\nu +1;c+1;x^{\alpha }\right)  \\
&+\frac{3\nu \left( \nu +1\right) x^{2\alpha }}{c\left( c+1\right) }%
~_{2}F_{1}\left( \mu +2,\nu +2;c+2;x^{\alpha }\right)  \\
&+\frac{\nu \left( \nu +1\right) \left( \nu +2\right) x^{3\alpha }}{c\left(
c+1\right) \left( c+2\right) }~_{2}F_{1}\left( \mu +3,\nu +3;c+3;x^{\alpha
}\right) .
\end{split}
\end{equation}

Relation \eqref{eq9.12} can be written in the form
\begin{equation*}
_{2}F_{1}\left( \mu +3,\nu ;c;x^{\alpha }\right)
=~\sum\limits_{k=0}^{3}\left( 
\begin{array}{c}
3 \\ 
k
\end{array}
\right) \frac{\left( \nu \right) _{k}~}{\left( c\right) _{k}}x^{\alpha
k}~_{2}F_{1}\left( \mu +k,\nu +k;c+k;x^{\alpha }\right) .
\end{equation*}

In general, we may write that 
\begin{equation*}
_{2}F_{1}\left( \mu +n,\nu ;c;x^{\alpha }\right)
=~\sum\limits_{k=0}^{n}\left( 
\begin{array}{c}
n \\ 
k
\end{array}
\right) \frac{\left( \nu \right) _{k}~}{\left( c\right) _{k}}x^{\alpha
k}~_{2}F_{1}\left( \mu +k,\nu +k;c+k;x^{\alpha }\right) .
\end{equation*}

In order to prove \eqref{eq9.8}, we note from  \eqref{eq9.2} of theorem \ref{tR1} $
\left( \text{with }n=1\right) $ that 
\begin{equation}\label{eq9.13}
_{2}F_{1}\left( \mu -1,\nu ;c;x^{\alpha }\right) =~_{2}F_{1}\left( \mu ,\nu
;c;x^{\alpha }\right) -\frac{\nu x^{\alpha }}{c}~_{2}F_{1}\left( \mu ,\nu
+1;c+1;x^{\alpha }\right) .
\end{equation}

Similarly, $\left( \text{with }
n=2\right) $ yields 
\begin{equation}\label{eq9.14}
\begin{split}
_{2}F_{1}\left( \mu -2,\nu ;c;x^{\alpha }\right)  =&~_{2}F_{1}\left( \mu
,\nu ;c;x^{\alpha }\right) -\frac{\nu x^{\alpha }}{c}~_{2}F_{1}\left( \mu
,\nu +1;c+1;x^{\alpha }\right)  \\
&-\frac{\nu x^{\alpha }}{c}~_{2}F_{1}\left( \mu -1,\nu +1;c+1;x^{\alpha
}\right) .
\end{split}
\end{equation}

Inserting \eqref{eq9.13} in \eqref{eq9.14}, we get 
\begin{equation}\label{eq9.15}
\begin{split}
_{2}F_{1}\left( \mu -2,\nu ;c;x^{\alpha }\right)  =&~_{2}F_{1}\left( \mu
,\nu ;c;x^{\alpha }\right) -\frac{2\nu x^{\alpha }}{c}~_{2}F_{1}\left( \mu
,\nu +1;c+1;x^{\alpha }\right)  \\
&+\frac{\nu \left( \nu +1\right) x^{2\alpha }}{c\left( c+1\right) }
~_{2}F_{1}\left( \mu ,\nu +2;c+2;x^{\alpha }\right) .
\end{split}
\end{equation}

Using the Pochhammer symbol we may write \eqref{eq9.15} as
\begin{equation*}
_{2}F_{1}\left( \mu -2,\nu ;c;x^{\alpha }\right)
=~\sum\limits_{k=0}^{2}\left( -1\right) ^{k}\left( 
\begin{array}{c}
2 \\ 
k
\end{array}
\right) \frac{\left( \nu \right) _{k}~}{\left( c\right) _{k}}x^{\alpha
k}~_{2}F_{1}\left( \mu ,\nu +k;c+k;x^{\alpha }\right) .
\end{equation*}

Thus, in general, we may write 
\begin{equation*}
_{2}F_{1}\left( \mu -n,\nu ;c;x^{\alpha }\right) =~\sum\limits_{k=0}^{n}\left( -1\right) ^{k}\left( 
\begin{array}{c}
n \\ 
k
\end{array}
\right) \frac{\left( \nu \right) _{k}~}{\left( c\right) _{k}}x^{\alpha
k}~_{2}F_{1}\left( \mu ,\nu +k;c+k;x^{\alpha }\right) ,
\end{equation*}
just as required in \eqref{eq9.8}.
\end{proof}

\begin{thm}\label{tR3}
For $\alpha \in (0,1],$ the following recursion formulas hold true for the
CFGHF, $_{2}F_{1}\left( \mu ,\nu ;c;x^{\alpha }\right) $ 
\begin{equation}\label{eq9.16}
\begin{split}
_{2}F_{1}\left( \mu ,\nu ;c+n;x^{\alpha }\right)  =&~\frac{\left( c\right)
_{n}}{\left( c-\nu \right) _{n}}\sum\limits_{k=0}^{n}\left( -1\right)
^{k}\left( 
\begin{array}{c}
n \\ 
k
\end{array}
\right) \frac{\left( \nu \right) _{k}~}{\left( c\right) _{k}}
~_{2}F_{1}\left( \mu ,\nu +k;c+k;x^{\alpha }\right) , \\
&~\left( \left\vert x^{\alpha }\right\vert <1,c+n\notin 
\mathbb{Z}
_{0}^{-},~n\in 
\mathbb{N}
_{0}=%
\mathbb{N}
\cup \left\{ 0\right\} \right) 
\end{split}
\end{equation}
\end{thm}
\begin{proof}
In view of \eqref{eq8.4}, we have 
\begin{equation*}
_{2}F_{1}\left( \mu ,\nu ;c+n;x^{\alpha }\right)  =\frac{\Gamma \left( c+n\right) }{\Gamma \left( \nu \right) \Gamma \left(
c+n-\nu \right) }\int\limits_{0}^{1}\tau ^{\nu -1}\left( 1-\tau \right)
^{c-\nu -1}\left( 1-x^{\alpha }\tau \right) ^{-\mu }.\left( 1-\tau \right)
^{n}d\tau 
\end{equation*}

Using the binomial theorem, we obtain 
\begin{equation}\label{eq9.17}
_{2}F_{1}\left( \mu ,\nu ;c+n;x^{\alpha }\right) =\frac{\Gamma \left(
c+n\right) }{\Gamma \left( \nu \right) \Gamma \left( c+n-\nu \right) }%
\int\limits_{0}^{1}\sum\limits_{k=0}^{n}\left( -1\right) ^{k}\left( 
\begin{array}{c}
n \\ 
k
\end{array}
\right) \tau ^{\nu +k-1}\left( 1-\tau \right) ^{c-\nu -1}\left( 1-x^{\alpha
}\tau \right) ^{-\mu }d\tau 
\end{equation}

Using the definition of the Pochhammer symbol, we may write \eqref{eq9.17} as 
\begin{eqnarray*}
_{2}F_{1}\left( \mu ,\nu ;c+n;x^{\alpha }\right)  &=&\frac{\left( c\right)
_{n}}{\left( c-\nu \right) _{n}}\sum\limits_{k=0}^{n}\left( -1\right)
^{k}\left( 
\begin{array}{c}
n \\ 
k
\end{array}
\right) \frac{\left( \nu \right) _{k}~}{\left( c\right) _{k}}.\frac{\Gamma
\left( c+k\right) }{\Gamma \left( \nu +k\right) \Gamma \left( c-\nu \right) }
\\
&&.\int\limits_{0}^{1}\tau ^{\nu +k-1}\left( 1-\tau \right) ^{c+k-\nu
-k-1}\left( 1-x^{\alpha }\tau \right) ^{-\mu }d\tau 
\end{eqnarray*}

Applying \eqref{eq8.4}, we obtain 
\begin{equation*}
_{2}F_{1}\left( \mu ,\nu ;c+n;x^{\alpha }\right) =~\frac{\left( c\right) _{n}%
}{\left( c-\nu \right) _{n}}\sum\limits_{k=0}^{n}\left( -1\right) ^{k}\left( 
\begin{array}{c}
n \\ 
k
\end{array}
\right) \frac{\left( \nu \right) _{k}~}{\left( c\right) _{k}}%
~_{2}F_{1}\left( \mu ,\nu +k;c+k;x^{\alpha }\right) ,
\end{equation*}
just as required in theorem \ref{tR3}.
\end{proof}
\section{Fractional Laplace transform of the CFGHF}
In \cite{abdeljawad2015conformable}, Abdeljawad defined the fractional Laplace transform in the conformable sense as follows:
\begin{defn} \cite{abdeljawad2015conformable}
Let $\alpha \in (0,1]$  and $f:[0,\infty )\rightarrow \mathbb{R} $  be real valued function. Then the fractional Laplace transform of
order $\alpha $  is defined by

\begin{equation}\label{eq10.1}
L_{\alpha }\left[ f\left( t\right) \right] =F_{\alpha }\left(
s\right) =\int\limits_{0}^{\infty }e^{-s\left( \frac{t^{\alpha }}{\alpha }
\right) }f\left( t\right) ~d_{\alpha }t=\int\limits_{0}^{\infty }e^{-s\left( 
\frac{t^{\alpha }}{\alpha }\right) }f\left( t\right) t^{\alpha -1}dt.
\end{equation}
\end{defn}
\begin{rem}
If $\alpha =1,$ then \eqref{eq10.1} is the classical definition of the Laplace transform of integer order.
\end{rem}

Also, the author in \cite{abdeljawad2015conformable} gave the following interesting results.
\begin{lem} \cite{abdeljawad2015conformable}
Let $\alpha \in (0,1]$  and $f:[0,\infty )\rightarrow \mathbb{R} $  be real valued function such that \\ $L_{\alpha }\left[
f\left( t\right) \right] =F_{\alpha }\left( s\right) $ exist. Then 
$ F_{\alpha }\left( s\right) = L\left[ f\left( \alpha t\right) ^{\frac{1}{\alpha }}\right],$ 
where $L\left[ f\left( t\right) 
\right] =\int\limits_{0}^{\infty }e^{-st}f\left( t\right) dt.$
\end{lem}
\begin{lem}\label{Laplace}\cite{abdeljawad2015conformable}
The following the conformable fractional Laplace transform of certain functions:
\begin{multicols}{2}
\begin{enumerate}
\item[(1)] $L_{\alpha }\left[ 1\right] =\frac{1}{s};~\ s>0$
\item[(2)] $L_{\alpha }\left[ t^{p}\right] =\alpha ^{\frac{p}{
\alpha }}\frac{\Gamma \left( 1+\frac{p}{\alpha }\right) }{s^{1+\frac{p}{
\alpha }}};~\ s>0$ \newline

\item[(3)] $L_{\alpha }\left[ e^{k\frac{t^{\alpha }}{\alpha }}\right] =\frac{
1}{s-k}$
\end{enumerate}
\end{multicols}
\end{lem}

Owing to the definition of CFGHF and applying  the conformable fractional Laplace transform operator of   an arbitrary order $\gamma \in (0,1]$, we have
\begin{equation}\label{eq10.2}
L_{\gamma }\left[ ~_{2}F_{1}\left( \mu ,\nu ;c;x^{\alpha }\right) 
\right] =L_{\gamma }\left[ \sum\limits_{n=0}^{\infty }\frac{
\left( \mu \right) _{n}\left( \nu \right) _{n}}{\left( c\right) _{n}n!}
x^{\alpha n}\right] =\sum\limits_{n=0}^{\infty }\frac{\left( \mu \right)
_{n}\left( \nu \right) _{n}}{\left( c\right) _{n}n!} L_{\gamma
}\left\{ x^{\alpha n}\right\} 
\end{equation}

Using (2) of lemma \ref{Laplace}, we obtain
\begin{equation}\label{eq10.3}
L_{\gamma }\left[ ~_{2}F_{1}\left( \mu ,\nu ;c;x^{\alpha }\right) 
\right] =\sum\limits_{n=0}^{\infty }\frac{\left( \mu \right) _{n}\left( \nu
\right) _{n}}{\left( c\right) _{n}n!}\gamma ^{\frac{n\alpha }{\gamma }}\frac{
\Gamma \left( 1+\frac{n\alpha }{\gamma }\right) }{s^{1+\frac{n\alpha }{
\gamma }}}
\end{equation}
\begin{rem}
If $\gamma =\alpha $ in \eqref{eq10.3} we have
\begin{equation}\label{eq10.4}
L_{\alpha }\left[ ~_{2}F_{1}\left( \mu ,\nu ;c;x^{\alpha }\right) 
\right] =\sum\limits_{n=0}^{\infty }\frac{\left( \mu \right) _{n}\left( \nu
\right) _{n}}{\left( c\right) _{n}n!}\frac{\alpha ^{n}\Gamma \left(
1+n\right) }{s^{1+n}}=\sum\limits_{n=0}^{\infty }\frac{\alpha ^{n}\left( \mu
\right) _{n}\left( \nu \right) _{n}}{\left( c\right) _{n}s^{1+n}}.
\end{equation}
\end{rem}
\begin{thm}
Let $\alpha \in (0,1]$ and $_{2}F_{1}\left( \mu ,\nu ;c;x^{\alpha }\right) $
be a conformable fractional hypergeometric function, then we have
\begin{equation}\label{eq10.5}
L_{\alpha }\left[ ~_{2}F_{1}\left( \mu ,\nu ;1;x^{\alpha }\left(
1-e^{-\frac{t^{\alpha }}{\alpha }}\right) \right) \right] =\frac{1}{s}
~_{2}F_{1}\left( \mu ,\nu ;s+1;x^{\alpha }\right)
\end{equation}
\end{thm}
\begin{proof}
Using \eqref{eq3.3} and \eqref{eq10.1}, one can see  
\begin{eqnarray}\label{eq10.6}
\begin{aligned}
L_{\alpha }\left[ ~_{2}F_{1}\left( \mu ,\nu ;1;x^{\alpha }\left(
1-e^{-\frac{t^{\alpha }}{\alpha }}\right) \right) \right] &= L_{\alpha }\left[ ~\sum\limits_{n=0}^{\infty }\frac{\left( \mu \right)
_{n}\left( \nu \right) _{n}}{\left( 1\right) _{n}n!}x^{\alpha n}\left( 1-e^{-
\frac{t^{\alpha }}{\alpha }}\right) ^{n}\right] \\
&=\sum\limits_{n=0}^{\infty }\frac{\left( \mu \right) _{n}\left( \nu
\right) _{n}}{n!}x^{\alpha n} L_{\alpha }\left[ ~\frac{1}{n!}
\left( 1-e^{-\frac{t^{\alpha }}{\alpha }}\right) ^{n}\right]
\end{aligned}
\end{eqnarray}

 But
\begin{eqnarray*}
L_{\alpha }\left[ ~\frac{1}{n!}\left( 1-e^{-\frac{
t^{\alpha }}{\alpha }}\right) ^{n}\right] = L_{\alpha }\left[ ~
\frac{1}{n!}\sum\limits_{k=0}^{n}\frac{\left( -n\right) _{k}}{k!}e^{-k\frac{
t^{\alpha }}{\alpha }}\right] 
=\frac{1}{n!}\sum\limits_{k=0}^{n}\frac{\left( -n\right) _{k}}{k!}
L_{\alpha }\left\{ e^{-k\frac{t^{\alpha }}{\alpha }}\right\}
\end{eqnarray*}

Using (3) of lemma \ref{Laplace}, we have
\begin{equation*}
L_{\alpha }\left[ ~\frac{1}{n!}\left( 1-e^{-\frac{
t^{\alpha }}{\alpha }}\right) ^{n}\right] =\frac{1}{n!}\sum\limits_{k=0}^{n}
\frac{\left( -n\right) _{k}}{k!}\frac{1}{s+k}
\end{equation*}

Since $\left( -n\right) _{k}=0$ if $k>n,$ then we can write
\begin{equation}\label{eq10.7}
L_{\alpha }\left[ ~\frac{1}{n!}\left( 1-e^{-\frac{
t^{\alpha }}{\alpha }}\right) ^{n}\right] =\frac{1}{n!}\sum\limits_{k=0}^{
\infty }\frac{\left( -n\right) _{k}}{k!\left( s+k\right) }
\end{equation}

Using $\frac{\left( s\right) _{k}}{s\left( s+1\right) _{k}}=\frac{1}{s+k},$ \eqref{eq10.7} becomes
\begin{eqnarray}\label{eq10.8}
\begin{aligned}
L_{\alpha }\left[ ~\frac{1}{n!}\left( 1-e^{-\frac{t^{\alpha }}{
\alpha }}\right) ^{n}\right]  &=\frac{1}{n!}\sum\limits_{k=0}^{\infty }
\frac{\left( -n\right) _{k}\left( s\right) _{k}}{s\left( s+1\right) _{k}k!}=
\frac{1}{s.n!}~_{2}F_{1}\left( -n,s;s+1;1\right)  \\
&=\frac{1}{s.n!}\frac{\left( 1\right) _{n}}{\left( s+1\right) _{n}}=\frac{1
}{s\left( s+1\right) _{n}}
\end{aligned}
\end{eqnarray}

Substituting \eqref{eq10.8} into \eqref{eq10.6}, we have 
\begin{eqnarray*}
L_{\alpha }\left[ ~_{2}F_{1}\left( \mu ,\nu ;1;x^{\alpha }\left(
1-e^{-\frac{t^{\alpha }}{\alpha }}\right) \right) \right] 
&=&\sum\limits_{n=0}^{\infty }\frac{\left( \mu \right) _{n}\left( \nu
\right) _{n}}{n!}\frac{1}{s\left( s+1\right) _{n}}x^{\alpha n} \\
&=&\frac{1}{s}
~_{2}F_{1}\left( \mu ,\nu ;s+1;x^{\alpha }\right) 
\end{eqnarray*}
as required.
\end{proof}

\begin{thm}
Let $\alpha \in (0,1]$ and $_{2}F_{1}\left( \mu ,\nu ;c;x^{\alpha }\right) $
be a conformable fractional hypergeometric function, then
\begin{equation}\label{eq10.9}
L_{\alpha }\left[ ~t^{\alpha n}\sin \left( at^{\alpha }\right) 
\right] =\frac{a\alpha ^{n+1}\Gamma \left( n+2\right) }{s^{n+2}}
~_{2}F_{1}\left( \frac{n+2}{2},\frac{n+3}{2};\frac{3}{2};-\left( \frac{
\alpha a}{s}\right) ^{2}\right) 
\end{equation}
\end{thm}
\begin{proof}
First, we see that  
\begin{eqnarray*}
L_{\alpha }\left[ ~t^{\alpha n}\sin \left( at^{\alpha
}\right) \right]  &=& L_{\alpha }\left[ ~t^{\alpha
n}\sum\limits_{k=0}^{\infty }\frac{\left( -1\right) ^{k}a^{2k+1}}{\left(
2k+1\right) !}t^{\alpha \left( 2k+1\right) }\right]  \\
&=&\sum\limits_{k=0}^{\infty }\frac{\left( -1\right) ^{k}a^{2k+1}}{\left(
2k+1\right) !} L_{\alpha }\left\{ t^{\alpha \left( n+2k+1\right)
}\right\} 
\end{eqnarray*}

Using (2) of lemma \ref{Laplace}, it follows that
\begin{eqnarray*}
L_{\alpha }\left[ ~t^{\alpha n}\sin \left( at^{\alpha
}\right) \right]  &=&\sum\limits_{k=0}^{\infty }\frac{\left( -1\right)
^{k}a^{2k+1}}{\left( 2k+1\right) !}\alpha ^{n+2k+1}\frac{\Gamma \left(
n+2k+2\right) }{s^{n+2k+2}} \\
&=&\frac{a\alpha ^{n+1}\Gamma \left( n+2\right) }{s^{n+2}}
\sum\limits_{k=0}^{\infty }\frac{\Gamma \left( n+2k+2\right) }{\Gamma \left(
n+2\right) \left( 2k+1\right) !}\left( \frac{ -\alpha ^{2}a^{2}}{s^{2}}
\right) ^{k} \\
&=&\frac{a\alpha ^{n+1}\Gamma \left( n+2\right) }{s^{n+2}}
\sum\limits_{k=0}^{\infty }\frac{\left( n+2\right) _{2k}}{\left( 2\right)
_{2k}}\left( \frac{-\alpha ^{2}a^{2}}{s^{2}}\right) ^{k}
\end{eqnarray*}
But $\left( n+2\right) _{2k}=\left( \frac{n+2}{2}\right)
_{k}.\left( \frac{n+3}{2}\right) _{k}$ and $\left( 2\right) _{2k}=\left(
1\right) _{k}.\left( \frac{3}{2}\right) _{k}=\left( \frac{3}{2}\right) _{k}k!
$.
Therefore
\begin{equation*}
\begin{split}
L_{\alpha }\left[ ~t^{\alpha n}\sin \left( at^{\alpha }\right) 
\right]  &=\frac{a\alpha ^{n+1}\Gamma \left( n+2\right) }{s^{n+2}}
\sum\limits_{k=0}^{\infty }\frac{\left( \frac{n+2}{2}\right) _{k}.\left( 
\frac{n+3}{2}\right) _{k}}{\left( \frac{3}{2}\right) _{k}k!}\left( \frac{ - \alpha ^{2}a^{2}}{s^{2}}\right) ^{k} \\
&=\frac{a\alpha ^{n+1}\Gamma \left( n+2\right) }{s^{n+2}}~_{2}F_{1}\left( 
\frac{n+2}{2},\frac{n+3}{2};\frac{3}{2};-\left( \frac{\alpha a}{s}\right)
^{2}\right). 
\end{split}
\end{equation*}
\end{proof}
\section{Applications}
The general solution of a wide class of conformable fractional differential equations of mathematical physics can be written in terms of the CFGHF after using a suitable change of independent variable. This phenomenon  will be illustrated through the following interesting discussion.

Abul-Ez et al. \cite{abul2020conformable} gave the hypergeometric representation of the conformable
fractional Legendre polynomials $P_{\alpha n}\left( x\right) $, as 
\begin{equation*}
P_{\alpha n}\left( x\right) =~_{2}F_{1}\left( -n,n+1;1;\frac{1-x^{\alpha }}{2
}\right) .
\end{equation*}
This formula can be easily obtained through the CFGHE as follows. \\
Note that, the conformable fractional Legendre polynomials $P_{\alpha n}\left( x\right) 
$ satisfy the conformable fractional differential equation  
\begin{equation}\label{eq11.15}
\left( 1-x^{2\alpha }\right) D_{x}^{\alpha }D_{x}^{\alpha }P_{\alpha
n}\left( x\right) -2\alpha x^{\alpha }D_{x}^{\alpha }P_{\alpha n}\left(
x\right) +\alpha ^{2}n\left( n+1\right) P_{\alpha n}\left( x\right) =0.
\end{equation}

With the help of $t^{\alpha }=\frac{1-x^{\alpha }}{2},$ we get 
\begin{equation*}
D_{x}^{\alpha }P_{\alpha n}=\left( \frac{-1}{2}\right) ~D_{t}^{\alpha
}P_{\alpha n}, \text{\ and }
D_{x}^{\alpha }D_{x}^{\alpha }P_{\alpha n}=\frac{1}{4}D_{t}^{\alpha
}D_{t}^{\alpha }P_{\alpha n}.
\end{equation*}

Using \eqref{eq11.15}, we obtain 
\begin{equation}\label{eq11.16}
t^{\alpha }\left( 1-t^{\alpha }\right) D_{t}^{\alpha }D_{t}^{\alpha
}P_{\alpha n}+\alpha \left\{ 1-2t^{\alpha }\right\} D_{t}^{\alpha }P_{\alpha
n}+\alpha ^{2}n\left( n+1\right) P_{\alpha n}=0.
\end{equation}

Comparing the last equation \eqref{eq11.16} with the CFGHE \eqref{eq3.1}, we obtain the parameters $\mu ,$ $\nu $ and $c$, such that 
\begin{equation*}
\mu =-n,~~\nu =n+1~\text{\ and }c=1,
\end{equation*}

Hence, we may write the conformable fractional Legendre polynomials as  
\begin{equation*}
\begin{split}
P_{\alpha n}\left( x\right)  &=~_{2}F_{1}\left( -n,n+1;1;t^{\alpha }\right) 
\\
&=~_{2}F_{1}\left( -n,n+1;1;\frac{1-x^{\alpha }}{2}\right) .
\end{split}
\end{equation*}

\begin{ex}
Consider the following conformable fractional differential equation 
\begin{equation}  \label{eq11.12}
\left( 1-e^{x^{\alpha }}\right) D_{x}^{\alpha }D_{x}^{\alpha }y+\frac{\alpha 
}{2}D_{x}^{\alpha }y+\alpha ^{2}e^{x^{\alpha }}y=0
\end{equation}

Then the general solution of \eqref{eq11.12} can be easily deduced as follows. \\
Let $t^{\alpha }=\left( 1-e^{x^{\alpha }}\right) $, then we have 
\begin{equation*}
D_{x}^{\alpha }y=-e^{x^{\alpha }}D_{t}^{\alpha }y=-\left( 1-t^{\alpha
}\right) D_{t}^{\alpha }y
\end{equation*}
and 
\begin{equation*}
D_{x}^{\alpha }D_{x}^{\alpha }y=-\alpha e^{x^{\alpha }}D_{t}^{\alpha
}y+e^{2x^{\alpha }}D_{t}^{\alpha }D_{t}^{\alpha }y =\left( 1-t^{\alpha }\right) ^{2}D_{t}^{\alpha }D_{t}^{\alpha }y-\alpha
\left( 1-t^{\alpha }\right) D_{t}^{\alpha }y
\end{equation*}

Now, in view of  \eqref{eq11.12}, it can be easily seen that, 
\begin{equation}  \label{eq11.13}
t^{\alpha }\left[ \left( 1-t^{\alpha }\right) ^{2}D_{t}^{\alpha
}D_{t}^{\alpha }y-\alpha \left( 1-t^{\alpha }\right) D_{t}^{\alpha }y\right]
-\frac{\alpha }{2}\left( 1-t^{\alpha }\right) D_{t}^{\alpha }y+\alpha
^{2}\left( 1-t^{\alpha }\right) y=0
\end{equation}

Simplifying \eqref{eq11.13}, we get 
\begin{equation}  \label{eq11.14}
t^{\alpha }\left( 1-t^{\alpha }\right) D_{t}^{\alpha }D_{t}^{\alpha
}y+\alpha \left\{ \frac{-1}{2}-t^{\alpha }\right\} D_{t}^{\alpha }y+\alpha
^{2}y=0
\end{equation}

Comparing \eqref{eq11.14} with the CFGHE \eqref{eq3.1}, we obtain $
\mu +\nu =0,~\mu \nu =-1 \text{\ and } c=\frac{-1}{2}$.
Thus,  $\mu =1$ and $\gamma =-1$. 
Therefore, the general solution of the CFDE \eqref{eq11.12} can be given in the form 
\begin{eqnarray*}
y &=&A~_{2}F_{1}\left( \mu ,\nu ;c;t^{\alpha }\right) +B~t^{\alpha \left(
1-c\right) }~_{2}F_{1}\left( \mu -c+1,\nu -c+1;2-c;t^{\alpha }\right) \\
&=&A~_{2}F_{1}\left( 1,-1;\frac{-1}{2};1-e^{x^{\alpha }}\right) +B~\left[
1-e^{x^{\alpha }}\right] ^{\frac{3}{2}}~_{2}F_{1}\left( \frac{5}{2},\frac{1}{
2};\frac{5}{2};1-e^{x^{\alpha }}\right)
\end{eqnarray*}%
where $A$ and $B$  are arbitrary constants.
\end{ex}
The strategy  used in the preceding  example can be easily applied to solve some famous differential equations such as, Chebyshev, Fibonacci, and Lucas differential equations in the framework of fractional calculus. 
Handled by Chebyshev, Fibonacci, and Lucas differential equations have advantages due to their own importance in applications. Thus, we may mention that the properties of Chebyshev polynomials are used to give a numerical solution of the conformable space-time fractional wave equation, see  \cite{yaslan2018numerical}.
The Fibonacci polynomial is a polynomial sequence, which can be considered as a generalization circular for the Fibonacci numbers. It is used in many applications, e.g., biology, statistics, physics, and computer science \cite{koshy2019fibonacci}.
The Fibonacci and Lucas sequences of both polynomials and numbers are of great importance in a variety of topics, such as number theory, combinatorics, and numerical analysis. For these studies, we refer to \cite{wang2015some,koshy2019fibonacci,taskara2010properties,gulec2013new}. Table \ref{Table} provides the general solutions of such famous differential equations briefly. 


\begin{table}[ht]
\centering
\begin{tabular}{|l|l|}
\hline
\multicolumn{2}{|c|}{\textbf{Conformable fractional Chebyshev differential
equation}} \\ \hline\hline
\textbf{CF Chebyshev DE} & $\left( 1-x^{2\alpha }\right) D_{x}^{\alpha }D_{x}^{\alpha
}y-\alpha x^{\alpha }D_{x}^{\alpha }y+\alpha ^{2}n^{2}y=0$ \\ \hline
\textbf{Suitable transformation} & $t^{\alpha }=\frac{1-x^{\alpha }}{2}$ \\ 
\hline
\textbf{Transformed equation} & $t^{\alpha }\left( 1-t^{\alpha }\right)
D_{t}^{\alpha }D_{t}^{\alpha }y+\alpha \left\{ \frac{1}{2}-t^{\alpha
}\right\} D_{t}^{\alpha }y+\alpha ^{2}n^{2}y=0.$ \\ \hline
\textbf{Parameters }$\left( \mu ,\nu \text{ and }c\right) $ & $\mu =-n,~~\nu
=n\text{\ and }c=\frac{1}{2}.$ \\ \hline
\textbf{General solution} & $y=A~_{2}F_{1}\left( -n,n;\frac{1}{2};\frac{%
1-x^{\alpha }}{2}\right) $\\
&\quad$+B~\left[ \frac{1-x^{\alpha }}{2}\right] ^{\frac{1}{%
2}}~_{2}F_{1}\left( -n+\frac{1}{2},n+\frac{1}{2};\frac{3}{2};\frac{%
1-x^{\alpha }}{2}\right) ,$
 \\ \hline\hline
\multicolumn{2}{|c|}{\textbf{Conformable fractional Fibonacci differential
equation}} \\ \hline\hline
\textbf{CF Fibonacci DE} & $\left( x^{2\alpha }+4\right) D_{x}^{\alpha }D_{x}^{\alpha
}y+3\alpha x^{\alpha }D_{x}^{\alpha }y-\alpha ^{2}\left( n^{2}-1\right) y=0$
\\ \hline
\textbf{Suitable transformation} & $t^{\alpha }=\left( 1+\frac{x^{2\alpha }}{%
4}\right) $ \\ \hline
\textbf{Transformed equation} & $t^{\alpha }\left( 1-t^{\alpha }\right)
D_{t}^{\alpha }D_{t}^{\alpha }y+\alpha \left\{ \frac{3}{2}-2t^{\alpha
}\right\} D_{t}^{\alpha }y-$\\ &\quad$\alpha ^{2}\frac{\left( 1-n^{2}\right) }{4}y=0.$
\\ \hline
\textbf{Parameters }$\left( \mu ,\nu \text{ and }c\right) $ & $\mu =\frac{1-n%
}{2},~~\nu =\frac{1+n}{2}~$\ and $c=\frac{3}{2}$ \\ \hline
\textbf{General solution} & $y=A~_{2}F_{1}\left( \frac{1-n}{2},\frac{1+n}{2};%
\frac{3}{2};1+\frac{x^{2\alpha }}{4}\right) $\\
&\quad$+B~\left[ 1+\frac{x^{2\alpha }}{4%
}\right] ^{\frac{-1}{2}}~_{2}F_{1}\left( \frac{-n}{2},\frac{n}{2};\frac{1}{2}%
;1+\frac{x^{2\alpha }}{4}\right) $ \\ \hline\hline
\multicolumn{2}{|c|}{\textbf{Conformable fractional Lucas differential
equation}} \\ \hline\hline
\textbf{CF Lucas DE} & $\left( x^{2\alpha }+4\right) D_{x}^{\alpha }D_{x}^{\alpha
}y+\alpha x^{\alpha }D_{x}^{\alpha }y-\alpha ^{2}n^{2}y=0,$ \\ \hline
\textbf{Suitable transformation} & $t^{\alpha }=\left( 1+\frac{x^{2\alpha }}{%
4}\right) $ \\ \hline
\textbf{Transformed equation} & $t^{\alpha }\left( 1-t^{\alpha }\right)
D_{t}^{\alpha }D_{t}^{\alpha }y+\alpha \left\{ \frac{1}{2}-t^{\alpha
}\right\} D_{t}^{\alpha }y+\alpha ^{2}\frac{n^{2}}{4}y=0.$ \\ \hline
\textbf{Parameters }$\left( \mu ,\nu \text{ and }c\right) $ & $\mu =\frac{n}{%
2},~~\nu =\frac{-n}{2}~\text{\ and }c=\frac{1}{2}$ \\ \hline
\textbf{General solution} & $y=A~_{2}F_{1}\left( \frac{n}{2},\frac{-n}{2};%
\frac{1}{2};1+\frac{x^{2\alpha }}{4}\right) $\\
&\quad$+B~\left[ 1+\frac{x^{2\alpha }}{4%
}\right] ^{\frac{1}{2}}~_{2}F_{1}\left( \frac{1+n}{2},\frac{1-n}{2};\frac{3}{%
2};1+\frac{x^{2\alpha }}{4}\right) $ \\ \hline
\end{tabular}
\caption{General Solutions of some famous CDEs}
\label{Table}
\end{table}

\section{Conclusion}

The Gaussian hypergeometric function $_{2}F_{1} \left(\mu , \nu ;c; x\right)$ has been studied extensively from its
mathematical point of view \cite{erdlyi1953higher}. This occurs probably, in part, due to its many applications on a large
variety of physical and mathematical problems. In quantum mechanics, the solution of the
Schr\"{o}dinger equation for some systems is expressed in terms of $_{2}F_{1}$ functions, as observed 
in solving the P\"{o}schl-Teller, Wood-Saxon, or Hulth\'{e}n en potentials \cite{flugge1971practical}. Another very important 
case is related to the angular momentum theory since the eigenfunctions of the angular
momentum operators are written in terms of $_{2}F_{1}$ functions \cite{abramowitz1972handbook}.
One important tool related to such problems  is then provided by the derivatives of the $_{2}F_{1}$ function with respect to the parameters $\mu , \nu $, and $c$ since they
allow one, for example, to write a Taylor expansion around given values $\mu _{0}, \nu _{0}$, or $c_{0}$.
Therefore, the importance of the Gaussian hypergeometric differential equation motivates one to give a detailed study on the CFGHF.  The solutions of the CFGHE are given to improve and generalize those given in \cite{hammad2019fractional}. Besides, many interesting properties, and useful formulas of CFGHF are presented. Finally, supported examples, showing that a class of conformable fractional differential equations of mathematical physics can be solved by means of the CFGHF.

It is interesting to mention that the obtained results of the current work has treated various famous aspects such as, generating functions, differential forms, contiguous relations, and recursion formulas, for which they have been generalized and developed in the context of the fractional setting. 
These aspects play important roles in themselves and their diverse applications. In fact, most of the special functions of mathematical physics and engineering,  for instance, the Jacobi and Laguerre polynomials can be expressed in terms of the Gauss hypergeometric function and other related hypergeometric functions. 
Therefore, the numerous generating functions involving extensions and generalizations of the Gauss hypergeometric function are capable  of playing important roles in the theory of special functions of applied mathematics and mathematical physics, see \cite{srivastava2014generating}.

The derivatives of any order of the GHF
$_{2}F_{1}( \mu ,\nu ;c; x )$ with respect to the parameters $\mu , \nu $, and $c$, which can be expressed  in
terms of generalizations of multivariable Kampe de F\'{e}riet functions, have many applications (see the work of \cite{ancarani2009derivatives}).
We may recall that applications of the contiguous function relation range from the evaluation of hypergeometric series to the derivation of the summation and transformation formulas for such series; these can be used to evaluate the contiguous functions to a hypergeometric function, see \cite{mubeen2014contiguous}.
Furthermore, using some contiguous function relations for the classical Gauss hypergeometric series $_{2}F_{1}$, several new recursion formulas for the Appell functions $F_{2}$ with important applications have been the subject of some research work, see for example \cite{opps2009recursion} and reference therein.
In conclusion, it is rather interesting to consider a wide generalization of the Gaussian hypergeometric function in the forthcoming work either in the framework of fractional calculus or in a higher dimensional setting. Our concluded results can be used for a wide variety of cases.

\end{document}